\documentclass[a4paper,11pt]{article}

\title{Permutation-based presentations for Brin's higher-dimensional
  Thompson groups~$nV$}
\author{Martyn Quick \\[7pt]
  Mathematical Institute, University of St Andrews,\\
  North Haugh, Fife, KY16 9SS, United Kingdom\\[3pt]
  \texttt{mq3@st-andrews.ac.uk}}

\usepackage{amsmath,amssymb}
\usepackage{bm}
\usepackage[margin=2.5cm]{geometry}
\usepackage{theorem}
\usepackage{tikz}
\usepackage{xcolor}

\usepackage{hyperref}
\hypersetup{colorlinks=true, linkcolor={red!50!black},
  citecolor={green!50!black}}

\allowdisplaybreaks

\theorembodyfont{\slshape}
\newtheorem{thm}{Theorem}[section]
\newtheorem{lemma}[thm]{Lemma}
\newtheorem{prop}[thm]{Proposition}
\newtheorem{cor}[thm]{Corollary}

\newcommand{\AND}{\qquad\text{and}\qquad}
\newcommand{\Cant}{\mathfrak{C}}

\newcommand{\defeq}{\mathrel{:=}}
\newcommand{\length}[1]{\mathopen{|}#1\mathclose{|}}
\newcommand{\modulus}[1]{\mathopen{|}#1\mathclose{|}}
\newcommand{\nbd}{\nobreakdash-}
\newcommand{\set}[2]{\{\,#1\mid#2\,\}}

\newcommand{\spc}{\vspace{\baselineskip}}
\newcommand{\swap}[2]{{(#1\;\;#2)}}
\DeclareMathOperator{\Sym}{Sym}
\DeclareMathOperator{\wt}{wt}

\renewcommand{\emptyset}{\varnothing}
\renewcommand{\epsilon}{\varepsilon}
\renewcommand{\geq}{\geqslant}
\renewcommand{\leq}{\leqslant}

\newcommand{\qed}{\hspace*{\fill}$\square$}
\newenvironment{proof}{%
  \begin{trivlist}
  \item\textsc{Proof:}}{\qed\end{trivlist}}

\renewcommand{\theenumi}{(\roman{enumi})}
\renewcommand{\labelenumi}{{\normalfont\theenumi}}

\makeatletter
\newcommand{\leqnomode}{\tagsleft@true}
\newcommand{\reqnomode}{\tagsleft@false}
\newenvironment{HMeqs}{%
  \global\chardef\dc@currentequation=\value{equation}%
  \let\c@equation\c@defcounter
  \align}{%
  \endalign
  \setcounter{equation}{\dc@currentequation}}
\makeatother

\begin{document}

\maketitle

\begin{abstract}
  The higher-dimensional Thompson groups~$nV$, for $n \geq 2$, were
  introduced by Brin in 2005.  We provide new presentations for each
  of these infinite simple groups.  The first is an infinite
  presentation, analogous to the Coxeter presentation for the finite
  symmetric group, with generating set equal to the set of
  transpositions in~$nV$ and reflecting the self-similar structure of
  $n$\nbd dimensional Cantor space.  We then exploit this infinite
  presentation to produce further finite presentations that are
  considerably smaller than those previously known.
\end{abstract}

\paragraph{Keywords:} Thompson's groups, higher-dimensional Thompson
groups, simple groups, Cantor space, presentations, generators and
relations, permutations, transpositions, baker's map

\paragraph{MSC Classification:} 20F65, 20F05, 20E32

\paragraph{Declarations of interest:} none

\section{Introduction}
\label{sec:intro}

A well-known result of Brouwer~\cite{Brouwer} states that a non-empty
totally disconnected compact metrizable space without isolated points
is homeomorphic to the Cantor space~$\Cant$.  Consequently this space
arises throughout mathematics and it is unsurprising that many groups
occur among its homeomorphisms.  Interesting and important examples of
such groups include Grigorchuk's group of intermediate
growth~\cite{Grig80,Grig91}, which may be naturally described as
consisting of certain automorphisms of a binary rooted tree, and various
generalizations, such as the Gupta--Sidki groups~\cite{GS83} and the
multi-GGS groups (see, for example,~\cite{FAZR}); the asynchronous
rational group of Grigorchuk, Nekrashevych and
Sushchanskii~\cite{GNS}; and, particularly relevant to this paper, the
groups $F$,~$T$ and~$V$ introduced by
Richard~J. Thompson~\cite{CFP,Thompson-notes}.

Thompson's group~$F$ is a $2$\nbd generator group with abelianization
isomorphic to a free abelian group of rank~$2$ and such that its
derived subgroup~$F'$ is simple.  The other two groups $T$~and~$V$
introduced by Thompson are both infinite simple groups.  All three
groups are finitely presented, with $F$~having a small presentation
with two generators and two relations.  The presentations for
$T$~and~$V$, as described in~\cite{CFP}, both involve additional
generators and relations to supplement those used for~$F$.  In
particular, Thompson's original presentation for his group~$V$
involved four generators and fourteen relations.  In work by Bleak and
the author~\cite{BQ}, we returned to possible presentations for~$V$.
We give there various presentations for this group: one involving
infinitely many generators and an infinite family of relations.  The
generators in~\cite[Theorem~1.1]{BQ} correspond to transpositions of
certain disjoint basic open sets of Cantor space and the relations are
analogous to the Coxeter presentation for a finite symmetric group,
but also include what we termed ``split relations'' reflecting the
self-similar structure of Cantor space.  The second presentation,
given in~\cite[Theorem~1.2]{BQ}, is a finite presentation, essentially
obtained by reducing the infinite presentation, with three generators
and eight relations (which compares favourably in size to Thompson's
original presentation).  We then produced a two-generator presentation
for~$V$ by use of Tietze transformations and our smallest
presentation, obtained via computational methods, is on two generators
and seven relations~\cite[Theorem~1.3]{BQ}.  One should also note a
link between our infinite presentation for~$V$ and the geometric
presentations given by Dehornoy~\cite{Dehornoy}.

In 2004, Brin~\cite{Brin-higherdim} introduced, for each positive
integer~$n \geq 2$, an analogue of Thompson's group~$V$ that acts upon
an $n$\nbd dimensional version of Cantor space.  He denotes this group
by~$nV$ and, via the homeomorphism $\Cant^{n} \cong \Cant$, this
family provide us with further groups of homeomorphisms of Cantor
space.  Bleak and Lanoue~\cite{BL} noted that two of these
groups $mV$~and~$nV$ are isomorphic if and only if $m = n$.  Brin
observes in his first paper that the group~$2V$ is an infinite simple
group, while in~\cite{Brin-baker} he shows that all the groups~$nV$
are simple.  In the latter argument, he makes considerable reference
to the baker's maps of~$\Cant^{n}$ and notes that these maps can be
expressed as a product of transpositions.  This observation will be
particularly relevant to our proof of Theorem~\ref{thm:infpres} in
Section~\ref{sec:infpres} below.  Furthermore, all these groups are
finitely presented: In~\cite[Theorem~5]{Brin-presentations}, Brin
observes that $2V$~has a finite presentation with $8$~generators and
$70$~relations.  This method was extended by Hennig and
Matucci~\cite{HM} to establish a finite presentation for the
groups~$nV$ involving $2n+4$~generators and $10n^{2} + 10n +
10$~relations (see~\cite[Theorem~25]{HM}).  Indeed, each group~$nV$ is
of type~$\mathrm{F}_{\infty}$, as established in~\cite{KMPN,FMWZ}.  In
this article, it is demonstrated that these groups possess infinite
presentations involving elements corresponding to transpositions of
disjoint basic open sets and involving relations that have a
Coxeter-like shape and reflect the self-similar nature of~$\Cant^{n}$
(see Theorem~\ref{thm:infpres} below).  Again these infinite
presentations bear comparison with Dehornoy's geometric
presentations~\cite{Dehornoy} for $F$~and~$V$.  In the final section
of the paper, we demonstrate that the group~$nV$ is isomorphic to a
group~$G$ with a finite presentation involving $3$~generators and
$2n^{2} + 3n + 11$~relations.  It is noteworthy that the number of
generators is bounded independent of the parameter~$n$ and that the
number of relations significantly improves upon the presentations
in~\cite{Brin-presentations,HM}.  In particular, the resulting finite
presentation for~$2V$ involves $3$~generators and $25$~relations.

\subsection{Notation}

We write~$\Cant$ for the Cantor set; that is, the collection of all
infinite words from the alphabet~$\{0,1\}$.  We also use the
set~$\{0,1\}^{\ast}$ of finite words in this alphabet and
write~$\epsilon$ to denote the empty word.  If $\alpha,\beta \in
\{0,1\}^{\ast}$, then $\alpha \preceq \beta$ indicates that
$\alpha$~is a prefix of~$\beta$.  On the other hand, the notation
$\alpha \perp \beta$ denotes that $\alpha \npreceq \beta$ and $\alpha
\nsucceq \beta$ and we then say these words are \emph{incomparable}.
The \emph{length} of a finite word~$\alpha$, denoted
by~$\length{\alpha}$, is the number of symbols from~$\{0,1\}$
occurring in~$\alpha$.

If $n$~is a positive integer with $n \geq 2$, the higher-dimensional
Thompson group~$nV$ is defined (see below) as consisting of certain
transformations defined on $n$\nbd dimensional Cantor space $\Gamma =
\Cant^{n}$.  Accordingly, we shall also need the set~$\Omega$ of
sequences $\bm{\alpha} = (\alpha_{1},\alpha_{2},\dots,\alpha_{n})$
where each~$\alpha_{i} \in \{0,1\}^{\ast}$ and we use the term
\emph{address} to refer to elements of~$\Omega$.  These addresses are
used to index the basic open subsets of~$\Gamma$ which in turn appear
in the definition of the elements of~$nV$.  We extend the concept of
incomparability to addresses by writing $\bm{\alpha} \perp
\bm{\beta}$, for a pair of addresses $\bm{\alpha} =
(\alpha_{1},\alpha_{2},\dots,\alpha_{n})$ and $\bm{\beta} =
(\beta_{1},\beta_{2},\dots,\beta_{n})$, when $\alpha_{d} \perp
\beta_{d}$ for some index~$d$ with $1 \leq d \leq n$.  Similarly, for
such addresses, we write $\bm{\alpha} \preceq \bm{\beta}$ when
$\alpha_{d} \preceq \beta_{d}$ for all $d = 1$,~$2$, \dots,~$n$.

The higher-dimensional Thompson group~$nV$ consists of certain
homeomorphisms of~$\Gamma$.  We shall use right action notation
throughout and so write~$\bm{w}g$ for the image of $\bm{w} \in \Gamma$
under $g \in nV$.  Let $\Gamma(\bm{\alpha}) = \set{\bm{\alpha
    w}}{\bm{w} \in \Gamma}$ be the collection of all sequences
in~$\Gamma$ with the address~$\bm{\alpha}$ as prefix and this is the
basic open set indexed by~$\bm{\alpha}$.  Note $\Gamma(\bm{\alpha})
\cap \Gamma(\bm{\beta}) = \emptyset$ if and only if $\bm{\alpha} \perp
\bm{\beta}$ and that $\Gamma(\bm{\alpha}) \supseteq
\Gamma(\bm{\beta})$ if and only if $\bm{\alpha} \preceq \bm{\beta}$.
An element~$g$ of~$nV$ is then described as follows: Given two
partitions $\Gamma = \bigcup_{i=1}^{k} \Gamma(\bm{\alpha}^{(i)}) =
\bigcup_{i=1}^{k} \Gamma(\bm{\beta}^{(i)})$ into the same number of
disjoint basic open sets, we define the homeomorphism~$g$ of~$\Gamma$
by $\bm{\alpha}^{(i)}\bm{w} \mapsto \bm{\beta}^{(i)}\bm{w}$ for $i =
1$,~$2$, \dots,~$k$ and any $\bm{w} \in \Gamma$.  Thus each
homeomorphism in~$nV$ is given by piecewise affine maps on~$\Gamma$
determined by two partitions of the space into the same number of
basic open sets and some bijection between the parts.
Figure~\ref{fig:partition} illustrates an example partition
of~$\Cant^{3}$; that is, a potential choice for domain or codomain
partially determining an element of~$3V$.  If
$\bm{\alpha}$~and~$\bm{\beta}$ are incomparable addresses, we call the
element of~$nV$ that maps $\bm{\alpha w} \mapsto \bm{\beta w}$,
\ $\bm{\beta w} \mapsto \bm{\alpha w}$ for all $\bm{w} \in \Gamma$ and
fixes all other points in~$\Gamma$ a \emph{transposition}.  This
element has the effect of interchanging the basic open sets
$\Gamma(\bm{\alpha})$~and~$\Gamma(\bm{\beta})$.  In
what follows, we shall write~$G_{\infty}$ for the group with the
presentation given in Theorem~\ref{thm:infpres} below.  The element
denoted by~$\swap{\bm{\alpha}}{\bm{\beta}}$ that appears in that
presentation corresponds, under the natural homomorphism $G_{\infty}
\to nV$, to this transposition that interchanges
$\Gamma(\bm{\alpha})$~and~$\Gamma(\bm{\beta})$.

\begin{figure}
  \begin{center}
    \begin{tikzpicture}
      \draw (0, 0, 4) -- (4, 0, 4);
      \draw (0, 4, 4) -- (4, 4, 4);
      \draw (0, 0, 4) -- (0, 4, 4);
      \draw (4, 0, 4) -- (4, 4, 4);
      \draw (4, 0, 4) -- (4, 0, 0);
      \draw (4, 4, 4) -- (4, 4, 0);
      \draw (0, 4, 4) -- (0, 4, 0);
      \draw (0, 4, 0) -- (4, 4, 0);
      \draw (4, 0, 0) -- (4, 4, 0);
      \draw (0, 1, 4) -- (4, 1, 4) -- (4, 1, 0);
      \draw (0, 2, 4) -- (4, 2, 4) -- (4, 2, 0);
      \draw (2, 2, 4) -- (2, 4, 4) -- (2, 4, 3) -- (4, 4, 3) -- (4, 2, 3);
      \draw (2, 4, 3) -- (2, 4, 2) -- (4, 4, 2) -- (4, 2, 2);
      \draw (3, 4, 2) -- (3, 4, 0);
      \draw (2, 4, 2) -- (2, 4, 0);
      \node at (2,.5,4) {$\bm{00}_{3}$};
      \node at (2,1.5,4) {$\bm{01}_{3}$};
      \node at (1,3,4) {$\bm{0}_{1}.\bm{1}_{3}$};
      \node at (3,3,4) {$\bm{1}_{1}.\bm{00}_{2}.\bm{1}_{3}$};
      \node at (3,4,2.5) {$\bm{1}_{1}.\bm{01}_{2}.\bm{1}_{3}$};
      \node[above right] at (4,3.3,2.5) {$\bm{11}_{1}.\bm{1}_{2}.\bm{1}_{3}$};
      \node at (2.5,4,1) {$\bm{10}_{1}.\bm{1}_{2}.\bm{1}_{3}$};
    \end{tikzpicture}
  \end{center}
  \caption{A domain or codomain partition of~$\Gamma = \Cant^{3}$}
  \label{fig:partition}
\end{figure}

To describe an address in~$\Omega$ in theory requires one to write a
sequence of $n$~finite words in~$\{0,1\}$.  Such a sequence would
appear quite cumbersome in our calculations particularly when
appearing as entries in the transpositions that we work with.
Accordingly, we present a more compact and useful notation.  If
$\alpha$~is some (usually explicit) finite word in~$\{0,1\}$, we shall
write~$\bm{\alpha}_{d}$ for the address all of whose entries are the
empty word with the exception of the $d$th coordinate which
equals~$\alpha$.  Thus, for example, $\bm{010}_{d} =
(\epsilon,\dots,\epsilon,010,\epsilon,\dots,\epsilon)$ where
$010$~occurs in the $d$th coordinate in this $n$\nbd tuple.  We shall
particularly make use of this notation when we wish to append one (or
more) letters from~$\{0,1\}$ to particular entries in an address
$\bm{\alpha} = (\alpha_{1},\alpha_{2},\dots,\alpha_{n})$.  For
example, we write $\bm{\alpha}.\bm{0}_{d}$ to indicate that we
concatenate the addresses $\bm{\alpha}$~and~$\bm{0}_{d}$; that is, we
append the symbol~$0$ to the $d$th coordinate~$\alpha_{d}$
of~$\bm{\alpha}$:
\[
\bm{\alpha}.\bm{0}_{d} =
(\alpha_{1},\dots,\alpha_{d-1},\alpha_{d}0,\alpha_{d+1},\dots,\alpha_{n})
\]
(The use of the dot appearing this notation is to demarcate the end of
the first address~$\bm{\alpha}$ and the beginning of the second and is
intended to achieve clarity.  Indeed, according to our notation,
$\bm{\alpha0}_{d}$ (without the dot) would indicate the address with a
single non-empty entry~$\alpha0$ in the $d$th coordinate.  The dot
notation is unnecessary when concatenating two finite words in~$\{0,1\}$
but helps when dealing with $n$\nbd tuples.)  The use of this notation
can be observed within what we term the ``split relations'' appearing
in the statement of Theorem~\ref{thm:infpres} below (see
Equation~\eqref{rel:split}) and in the addresses labelling the parts
in Figure~\ref{fig:partition}.  An additional piece of notation that
we shall use is that if $x \in \{0,1\}$, then $\bar{x}$~denotes the
other element in this set and then, following our above convention,
$\bar{\bm{x}}_{d}$~is the sequence
$(\epsilon,\dots,\epsilon,\bar{x},\epsilon,\dots,\epsilon)$ where
$\bar{x}$~occurs in the $d$th coordinate.  Finally,
$\bm{\epsilon}$~will denote the address
$(\epsilon,\epsilon,\dots,\epsilon)$ all of whose entries are the
empty word.

To specify the relations that define our group, we define an
additional notation that encodes the partial action of transpositions
in~$nV$ upon the basic open sets indexed by the addresses
in~$\Omega$.  To be specific, if $\bm{\alpha},\bm{\beta},\bm{\gamma}
\in \Omega$ with $\bm{\alpha} \perp \bm{\beta}$, we define a partial
map by
\begin{equation}
  \bm{\gamma} \bullet \swap{\bm{\alpha}}{\bm{\beta}} =
  \begin{cases}
    \bm{\beta}\bm{\delta} &\text{if $\bm{\gamma} =
      \bm{\alpha}\bm{\delta}$ for some $\bm{\delta} \in \Omega$;} \\
    \bm{\alpha}\bm{\delta} &\text{if $\bm{\gamma} =
      \bm{\beta}\bm{\delta}$ for some $\bm{\delta} \in \Omega$;} \\
    \bm{\gamma} &\text{if both $\bm{\gamma} \perp \bm{\alpha}$ and
      $\bm{\gamma} \perp \bm{\beta}$;} \\
    \text{undefined} &\text{otherwise}.
  \end{cases}
  \label{eq:actiondef}
\end{equation}
Thus we associate to the symbol~$\swap{\bm{\alpha}}{\bm{\beta}}$ the
partial map on the set~$\Omega$ of addresses that performs a prefix
substitution that interchanges the prefix~$\bm{\alpha}$ with the
prefix~$\bm{\beta}$.

\subsection{Statement of results}

In Section~\ref{sec:infpres}, we shall establish the following
infinite presentation for~$nV$.

\begin{thm}
  \label{thm:infpres}
  Let $n \geq 2$.  Let $\mathcal{A}$~be the set of all
  symbols~$\swap{\bm{\alpha}}{\bm{\beta}}$ where
  $\bm{\alpha}$~and~$\bm{\beta}$ are addresses in\/~$\Omega$ with
  $\bm{\alpha} \perp \bm{\beta}$.  Then Brin's higher-dimensional
  Thompson group~$nV$ has infinite presentation with generating
  set~$\mathcal{A}$ and all relations
  \begin{align}
    \swap{\bm{\alpha}}{\bm{\beta}}^{2} &= 1
    \label{rel:order} \\*
    \swap{\bm{\alpha}}{\bm{\beta}}^{\swap{\bm{\gamma}}{\bm{\delta}}}
    &= \swap{\bm{\alpha}\!\bullet\!\swap{\bm{\gamma}}{\bm{\delta}}}{%
      \bm{\beta}\!\bullet\!\swap{\bm{\gamma}}{\bm{\delta}}}
    \label{rel:conj} \\*
    \swap{\bm{\alpha}}{\bm{\beta}} &=
    \swap{\bm{\alpha}.\bm{0}_{d}}{\bm{\beta}.\bm{0}_{d}} \;
    \swap{\bm{\alpha}.\bm{1}_{d}}{\bm{\beta}.\bm{1}_{d}}
    \label{rel:split}
  \end{align}
  where $\bm{\alpha}$,~$\bm{\beta}$, $\bm{\gamma}$ and~$\bm{\delta}$
  range over all addresses in~$\Omega$ such that $\bm{\alpha} \perp
  \bm{\beta}$, \ $\bm{\gamma} \perp \bm{\delta}$ and such that both
  $\bm{\alpha} \bullet \swap{\bm{\gamma}}{\bm{\delta}}$ and
  $\bm{\beta} \bullet \swap{\bm{\gamma}}{\bm{\delta}}$ are defined,
  and $d$~ranges over all indices with $1 \leq d \leq n$.
\end{thm}

We shall refer to relations of the form~\eqref{rel:conj} as
``conjugacy relations'' and those of the form~\eqref{rel:split} as
``split relations'' in what follows.  The latter arise due to the
self-similar nature of Cantor space: to exchange prefixes
$\bm{\alpha}$~and~$\bm{\beta}$ is equivalent to exchanging both the
pairs of prefixes obtained by ``splitting'' the $d$th coordinate.
Note that we use exponential notation for conjugation writing~$g^{h}$
for~$h^{-1}gh$ where $g$~and~$h$ belong to some group and this is
consistent with our use of right actions.

We shall also need an additional relation that can be deduced
immediately from~\eqref{rel:conj}.  On the face of it, if
$\bm{\alpha}$~and~$\bm{\beta}$ are incomparable addresses in~$\Omega$,
the symbols
$\swap{\bm{\alpha}}{\bm{\beta}}$~and~$\swap{\bm{\beta}}{\bm{\alpha}}$
are not necessarily the same element of the group with the given
presentation.  However, we expect them to correspond to the same
transposition in~$nV$.  The ``symmetry relation''
\begin{equation}
  \swap{\bm{\alpha}}{\bm{\beta}} = \swap{\bm{\beta}}{\bm{\alpha}}
  \label{rel:symm} \\
\end{equation}
follows from~\eqref{rel:conj} by taking $\bm{\gamma} = \bm{\alpha}$
and $\bm{\delta} = \bm{\beta}$:
\[
\swap{\bm{\alpha}}{\bm{\beta}} =
\swap{\bm{\alpha}}{\bm{\beta}}^{%
  \swap{\bm{\alpha}}{\bm{\beta}}} =
\swap{\bm{\alpha} \!\bullet\! \swap{\bm{\alpha}}{\bm{\beta}}}{%
  \bm{\beta} \!\bullet\! \swap{\bm{\alpha}}{\bm{\beta}}} =
\swap{\bm{\beta}}{\bm{\alpha}}.
\]
We shall make use of this additional relation~\eqref{rel:symm}
throughout our arguments.

The method of proof of the above theorem is essentially to verify a
family of relations for~$nV$ originally found in~\cite{HM}.  Let
$G_{\infty}$~denote the group with presentation given in
Theorem~\ref{thm:infpres}.  The key steps in the proof in
Section~\ref{sec:finitepres} are to define and investigate elements
in~$G_{\infty}$ that correspond to baker's maps.  The two-dimensional
baker's map is a basic object within the study of dynamical systems
(see, for example,~\cite{Devaney}) and is illustrated in
Figure~\ref{fig:baker}(i).  In the context of $n$~dimensions, we shall
refer to baker's maps that arise in the domain from a cut in the first
coordinate and in the codomain from a cut in the $d$th coordinate.
Thus we define, in~$G_{\infty}$, an element~$B_{d}(\bm{\alpha})$ that
corresponds to the element of~$nV$ with support equal
to~$\Gamma(\bm{\alpha})$ mapping $\Cant^{n} \to \Cant^{n}$ via the
formula
\[
\bm{w} \mapsto \begin{cases}
  \bm{\alpha}.\bm{0}_{d}.\bm{u} &\text{if $\bm{w} =
    \bm{\alpha}.\bm{0}_{1}.\bm{u}$ for some $\bm{u} \in \Cant^{n}$,} \\
  \bm{\alpha}.\bm{1}_{d}.\bm{u} &\text{if $\bm{w} =
    \bm{\alpha}.\bm{1}_{1}.\bm{u}$ for some $\bm{u} \in \Cant^{n}$,}
  \\
  \bm{w} &\text{otherwise (that is, if $\bm{\alpha} \npreceq
    \bm{w}$).}
\end{cases}
\]
When we refer below to the element~$B_{d}(\bm{\alpha})$ evaluating to
an ``index~$d$'' baker's map, we mean that it evaluates to the
homeomorphism of~$\Cant^{n}$ given by this formula.  All baker's maps
arising within our work will have such a form (for some address
$\bm{\alpha} \in \Omega$ and some $d \geq 2$).

\begin{figure}
  \begin{center}
    \begin{tabular}{cc}
      \begin{tikzpicture}
        \draw (0,0) -- (2,0) -- (2,2) -- (0,2) -- cycle;
        \draw (1,0) -- (1,2);
        \node at (.5,1) {$0$};
        \node at (1.5,1) {$1$};
        \draw[->] (2.5,1) -- (3.5,1);
        \draw (4,0) -- (6,0) -- (6,2) -- (4,2) -- cycle;
        \draw (4,1) -- (6,1);
        \node at (5,.5) {$0$};
        \node at (5,1.5) {$1$};
      \end{tikzpicture}
      &\qquad\quad
      \begin{tikzpicture}
        \draw (0,0,2) -- (2,0,2) -- (2,2,2) -- (0,2,2) -- cycle;
        \draw (2,0,2) -- (2,0,0) -- (2,2,0) -- (2,2,2);
        \draw (0,2,2) -- (0,2,0) -- (2,2,0);
        \draw (1,0,2) -- (1,2,2) -- (1,2,0);
        \node at (.5,1,2) {$0$};
        \node at (1.5,1,2) {$1$};
        \draw[->] (2.8,1,1) -- (4.1,1,1);
        \draw (4.9,0,2) -- (6.9,0,2) -- (6.9,2,2) -- (4.9,2,2) -- cycle;
        \draw (6.9,0,2) -- (6.9,0,0) -- (6.9,2,0) -- (6.9,2,2);
        \draw (4.9,2,2) -- (4.9,2,0) -- (6.9,2,0);
        \draw (4.9,2,1) -- (6.9,2,1) -- (6.9,0,1);
        \node at (5.9,1,2) {$0$};
        \node at (5.9,2,.5) {$1$};
      \end{tikzpicture}
    \end{tabular}
  \end{center}
  \caption{Baker's maps when (i)~$d = 2$ and (ii)~$d = 3$}
  \label{fig:baker}
\end{figure}

The elements~$B_{d}(\bm{\alpha})$ in~$G_{\infty}$ are defined via the
formulae expressing baker's maps in terms of transpositions found
in~\cite{Brin-baker}.  In Section~\ref{sec:infpres}, we observe that
the behaviour of baker's maps can be deduced from the relations
assumed about transpositions.
Lemmas~\ref{lem:prebaker}--\ref{lem:fullbaker} give the properties
upon which we depend.  In summary, while Brin~\cite{Brin-baker}
establishes simplicity of~$nV$ by expressing baker's maps as products
of transpositions, we use relational properties between transpositions
to produce information about baker's maps to establish our
presentation in Theorem~\ref{thm:infpres}.

In Section~\ref{sec:finitepres}, we reduce our infinite presentation
to a finite presentation (the relations are those listed
in~\ref{rel:R1-Sym}--\ref{rel:R7-c-def} in that section):

\begin{thm}
  \label{thm:finitepres}
  Let $n \geq 2$.  Brin's higher-dimensional Thompson group~$nV$ has a
  finite presentation with three generators and $2n^{2} + 3n + 11$
  relations.
\end{thm}

To prove this theorem, we begin with transpositions with entries from
the set~$\Delta$ of addresses $\bm{\alpha} =
(\alpha_{1},\alpha_{2},\dots,\alpha_{n})$ with $\length{\alpha_{d}} =
2$ for $1 \leq d \leq n$.  Thus, as a base point in an induction
argument we assume that we have a subgroup isomorphic to (a quotient
of) the symmetric group of degree~$4^{n}$.  In an induction argument,
we build further transpositions by successively conjugating the
transpositions constructed at a previous stage and then finally
exploit the split relations~\eqref{rel:split} to complete the
definitions.  In this way, we demonstrate that the group~$G$ with the
presentations provided in Theorem~\ref{thm:finitepres} is a quotient
of the group~$G_{\infty}$ described in Theorem~\ref{thm:infpres}.

Finally, by applying Tietze transformations, we shall deduce:

\begin{cor}
  \label{cor:2gen}
  Let $n \geq 2$.  Brin's higher-dimensional Thompson group~$nV$ has a
  finite presentation with two generators and $2n^{2} + 3n + 13$
  relations.
\end{cor}

\paragraph{Remarks:} In common with the finite presentation given by
Hennig--Matucci~\cite{HM}, the number of relations we use is quadratic
in the dimension~$n$.  One could ask whether there are presentations
for~$nV$ on two generators, but where the number of relations grows at
most linearly in~$n$?  In our case, the quadratic function arises from
the family of Relations~\ref{rel:R5-welldef} which is used to ensure
that the well-definedness of
transpositions~$\swap{\bm{\alpha}}{\bm{\beta}}$ where two coordinates
of~$\bm{\alpha}$ have length~$3$ and all remaining coordinates
of~$\bm{\alpha}$ and all those of~$\bm{\beta}$ have length~$2$.
Although the growth in the number of relations relative to the
dimension of the space acted upon seems reasonable, surprising results
such as that of Guralnick--Kantor--Kassabov--Lubotzky~\cite{GKKL}
stand in contrast to expectations.  The arguments used in~\cite{GKKL}
employ a process that has been termed the \emph{Burnside procedure}
and presented in detail in the Appendix of~\cite{MPMN}.  This process
can be used to reduce some large presentation of a group to a much
smaller one.  Potentially  it could be applied to the infinite
presentation found in Theorem~\ref{thm:infpres} and one might wonder
how the result would compare with Theorem~\ref{thm:finitepres}.  It
seemed to the author that the most direct application of the Burnside
procedure (if successful) would likely result in more relations.
Nevertheless, it remains an interesting question whether smaller
presentations exist for Brin's groups~$nV$.

\section{The infinite presentation for~\boldmath$nV$}
\label{sec:infpres}

We devote this section to establishing Theorem~\ref{thm:infpres}.
Accordingly we define $G_{\infty}$~to be the group presented by the
generators $\mathcal{A} = \set{ \swap{\bm{\alpha}}{\bm{\beta}}}{
  \bm{\alpha},\bm{\beta} \in \Omega, \; \bm{\alpha} \perp \bm{\beta}}$
subject to the family of
relations~\eqref{rel:order}--\eqref{rel:split}.  In this context, we
shall use the term \emph{transposition} for any
element~$\swap{\bm{\alpha}}{\bm{\beta}}$ appearing in the generating
set~$\mathcal{A}$.  It was observed by Brin~\cite{Brin-baker} that the
group~$nV$ is generated by the corresponding transpositions of basic
open sets of~$\Gamma$.  It is readily verified that these
homeomorphisms satisfy the relations listed in
Theorem~\ref{thm:infpres}.  Hence there exists a surjective
homomorphism $\phi \colon G_{\infty} \to nV$ that
maps~$\swap{\bm{\alpha}}{\bm{\beta}}$ to the corresponding
transposition in~$nV$.  In what follows, we shall speak of
\emph{evaluating} a product~$g$ in~$nV$ to mean the effect of applying
the homomorphism~$\phi$ to the element $g \in G_{\infty}$.

We can extend the definition appearing in
Equation~\eqref{eq:actiondef} to a product~$g$ of transpositions, say
$g = g_{1}g_{2}\dots g_{k}$ where each~$g_{i} \in \mathcal{A}$, by
defining~$\bm{\alpha}\bullet g$ to equal the value obtained by
successively applying Equation~\eqref{eq:actiondef} with each
transposition~$g_{i}$.  Note that this is strictly speaking a function
of the word in~$\mathcal{A}$ representing~$g$ rather than depending
upon~$g$ \emph{as an element of~$G_{\infty}$}.  With this extended
definition, if $\bm{\alpha},\bm{\beta} \in \Omega$ with $\bm{\alpha}
\perp \bm{\beta}$ and $g \in G_{\infty}$ is expressed as a product of
transpositions in such that both $\bm{\alpha} \bullet g$ and
$\bm{\beta} \bullet g$ are defined, then it follows by repeated use of
the conjugacy relations~\eqref{rel:conj} that
\[
\swap{\bm{\alpha}}{\bm{\beta}}^{g} = g^{-1} \,
\swap{\bm{\alpha}}{\bm{\beta}} \, g =
\swap{\bm{\alpha}\!\bullet\!g}{\bm{\beta}\!\bullet\!g}.
\]

Note that $\bm{\alpha} \bullet g$, when it is defined, coincides with
the value obtained if the product~$g$ is evaluated as an element of
the Brin's higher-dimensional Thompson group~$nV$ and then
$\bm{\alpha} \bullet g$~is calculated via the natural partial action
of~$nV$ upon the addresses~$\Omega$.  The only difference is that
there may exist some addresses~$\bm{\alpha}$ for which $\bm{\alpha}
\bullet g$~is not defined for our given word representing~$g$ but for
which the corresponding transformation in~$nV$ does have an action
defined upon~$\bm{\alpha}$.  However, if $\bm{\alpha} =
(\alpha_{1},\alpha_{2},\dots,\alpha_{n})$ and \emph{provided} the
words~$\alpha_{i}$ are sufficiently long, then $\bm{\alpha} \bullet
g$~is defined and hence coincides with the value obtained via the
partial action of~$nV$ upon~$\Omega$.

We shall establish that $G_{\infty}$~is isomorphic to the group~$nV$
by demonstrating that a family of relations found within
Hennig--Matucci's work~\cite{HM} can be deduced from our defining
relations.  We shall define elements of our group~$G_{\infty}$ that
correspond to the family of generators that are used in~\cite{HM}.
Some are readily constructed using
transpositions~$\swap{\bm{\alpha}}{\bm{\beta}}$ but others depend upon
building analogues of the baker's maps.  Brin's
paper~\cite{Brin-baker} describes how to construct a baker's map from
transpositions in his Lemma~3.  By following this recipe we are able
to define the required elements of~$G_{\infty}$.  Moreover, it then
follows that the products of transformations we define evaluate to the
required baker's maps in~$nV$ and hence we can determine the value
of~$\bm{\alpha} \bullet g$ for such products~$g$ provided the
coordinates of~$\bm{\alpha}$ are sufficiently long.

If $\bm{\alpha},\bm{\beta} \in \Omega$ with $\bm{\alpha} \perp
\bm{\beta}$ and $d$~is an index with $2 \leq d \leq n$, we define
\begin{equation}
  A_{d}(\bm{\alpha},\bm{\beta}) =
  \swap{\bm{\alpha}.\bm{0}_{1}}{\bm{\beta}.\bm{0}_{d}} \;
  \swap{\bm{\alpha}.\bm{1}_{1}}{\bm{\beta}.\bm{1}_{d}} \;
  \swap{\bm{\alpha}}{\bm{\beta}}.
  \label{eq:A-def}
\end{equation}
We define further elements of~$G_{\infty}$ in terms of this product as
follows:
\begin{equation}
  \begin{aligned}
    \hat{B}_{d}(\bm{0}_{1}.\bm{0}_{d}) &=
    A_{d}(\bm{0}_{1},\bm{\gamma}) \;\,
    \swap{\bm{0}_{1}.\bm{01}_{d}}{\bm{0}_{1}.\bm{10}_{d}} \;\,
    A_{d}(\bm{\gamma},\bm{0}_{1}.\bm{1}_{d}) \\*
    \hat{B}_{d}(\bm{0}_{1}.\bm{1}_{d}) &=
    A_{d}(\bm{0}_{1},\bm{\gamma}) \;\,
    \swap{\bm{0}_{1}.\bm{01}_{d}}{\bm{0}_{1}.\bm{10}_{d}} \;\,
    A_{d}(\bm{\gamma},\bm{0}_{1}.\bm{0}_{d}) \\*
    \hat{B}_{d}(\bm{1}_{1}.\bm{0}_{d}) &=
    A_{d}(\bm{1}_{1},\bm{\gamma}) \;\,
    \swap{\bm{1}_{1}.\bm{01}_{d}}{\bm{1}_{1}.\bm{10}_{d}} \;\,
    A_{d}(\bm{\gamma},\bm{1}_{1}.\bm{1}_{d}) \\*
    \hat{B}_{d}(\bm{1}_{1}.\bm{1}_{d}) &=
    A_{d}(\bm{1}_{1},\bm{\gamma}) \;\,
    \swap{\bm{1}_{1}.\bm{01}_{d}}{\bm{1}_{1}.\bm{10}_{d}} \;\,
    A_{d}(\bm{\gamma},\bm{1}_{1}.\bm{0}_{d})
  \end{aligned}
  \label{eq:hatB-def}
\end{equation}
where $\bm{\gamma} = (\gamma_{1},\gamma_{2},\dots,\gamma_{n})$ is an
address in~$\Omega$ satisfying
$\length{\gamma_{1}},\length{\gamma_{d}} \geq 1$ and the additional
condition that $\bm{\gamma} \perp \bm{0}_{1}$ in the first two
definitions and that $\bm{\gamma} \perp \bm{1}_{1}$ in the last two
definitions in~\eqref{eq:hatB-def}.  Further, we then define:
\begin{equation}
  \begin{aligned}
    B_{d}(\bm{0}_{1}) &= \hat{B}_{d}(\bm{0}_{1}.\bm{0}_{d}) \;\,
    \hat{B}_{d}(\bm{0}_{1}.\bm{1}_{d}) \;\,
    \swap{\bm{0}_{1}.\bm{01}_{d}}{\bm{0}_{1}.\bm{10}_{d}} \\*
    B_{d}(\bm{1}_{1}) &= \hat{B}_{d}(\bm{1}_{1}.\bm{0}_{d}) \;\,
    \hat{B}_{d}(\bm{1}_{1}.\bm{1}_{d}) \;\,
    \swap{\bm{1}_{1}.\bm{01}_{d}}{\bm{1}_{1}.\bm{10}_{d}}
  \end{aligned}
  \label{eq:B-def}
\end{equation}
These three types of element are the analogues of the maps arising
within the proof of~\cite[Lemma~3]{Brin-baker} and their definition
precisely follows that proof.  Consequently, the
product~$A_{d}(\bm{\alpha},\bm{\beta})$ evaluates in the group~$nV$ to
the composite of an ``index~$d$'' baker's map with
support~$\Gamma(\bm{\alpha})$ and the inverse of an ``index~$d$''
baker's map with support~$\Gamma(\bm{\beta})$.  The subsequent
elements $\hat{B}_{d}(\bm{\alpha})$~and~$B_{d}(\bm{\alpha})$ both
evaluate to the ``index~$d$'' baker's map with
support~$\Gamma(\bm{\alpha})$.  The difference is that the address
$\bm{\alpha} = (\alpha_{1},\alpha_{2},\dots,\alpha_{n})$ that we have
first defined them upon satisfies $\length{\alpha_{1}} = 1$ for both
products but the~$\hat{B}_{d}$ version requires $\length{\alpha_{d}} =
1$ while $B_{d}$~permits $\alpha_{d}$~to be empty.  One notes that to
define a baker's map on the whole space $\Gamma = \Cant^{n}$ (that is,
with address~$\bm{\epsilon}$) requires a further such definition.  As
this is (up to choice of index~$d$) a single element in~$G_{\infty}$,
we delay the definition of this element, which appears as~$C_{0,d}$
below (see Equation~\eqref{eq:HMgens1}).

To extend the baker's maps to arbitrary addresses we use another
convenient notation.  If $g$~is an element of~$G_{\infty}$ and
$\bm{\delta} \in \Omega$, we write $\bm{\delta}.g$ for the
element of~$G_{\infty}$ obtained by inserting~$\bm{\delta}$ as a
prefix in both entries of every transposition appearing in the
product~$g$.  Since the relations~\eqref{rel:order}--\eqref{rel:split}
are closed under performing such insertions, it follows that
(i)~$\bm{\delta}.g$~is a well-defined element of~$G_{\infty}$ and
(ii)~if $u = v$ is a relation that holds in~$G_{\infty}$ then
$\bm{\delta}.u = \bm{\delta}.v$ also holds in~$G_{\infty}$.  We shall
use the latter observation to reduce the number of calculations required.

In terms of this prefix notation, observe that
\[
\bm{\delta}.A_{d}(\bm{\alpha},\bm{\beta}) =
A_{d}(\bm{\delta}\bm{\alpha},\bm{\delta}\bm{\beta})
\]
for any addresses $\bm{\alpha}$,~$\bm{\beta}$ and~$\bm{\delta}$ with
$\bm{\alpha} \perp \bm{\beta}$.  For our baker's maps, we define
$\hat{B}_{d}(\bm{\alpha})$, where $\bm{\alpha} =
(\alpha_{1},\alpha_{2},\dots,\alpha_{n})$ satisfies
$\length{\alpha_{1}},\length{\alpha_{d}} \geq 1$, by writing
$\bm{\alpha} = \bm{\delta}\bm{\beta}$ where $\bm{\beta} \in \{
\bm{0}_{1}.\bm{0}_{d}, \bm{0}_{1}.\bm{1}_{d}, \bm{1}_{1}.\bm{0}_{d},
\bm{1}_{1}.\bm{1}_{d} \}$ and then setting
\[
\hat{B}_{d}(\bm{\alpha}) = \bm{\delta}.\hat{B}_{d}(\bm{\beta}),
\]
where $\hat{B}_{d}(\bm{\beta})$~is as defined in
Equation~\eqref{eq:hatB-def}.  Similarly if $\bm{\alpha} =
(\alpha_{1},\alpha_{2},\dots,\alpha_{n})$ satisfies
$\length{\alpha_{1}} \geq 1$, by writing $\bm{\alpha} =
\bm{\delta}\bm{\beta}$ where $\bm{\beta} \in \{ \bm{0}_{1}, \bm{1}_{1}
\}$, we define
\[
B_{d}(\bm{\alpha}) = \bm{\delta}.B_{d}(\bm{\beta}).
\]
Note that inserting this prefix~$\bm{\delta}$ into the
definition~\eqref{eq:B-def} yields
\begin{equation}
  B_{d}(\bm{\alpha}) = \hat{B}_{d}(\bm{\alpha}.\bm{0}_{d}) \;
  \hat{B}_{d}(\bm{\alpha}.\bm{1}_{d}) \;
  \swap{\bm{\alpha}.\bm{01}_{d}}{\bm{\alpha}.\bm{10}_{d}}
  \label{eq:B-form}
\end{equation}
for any $\bm{\alpha} = (\alpha_{1},\alpha_{2},\dots,\alpha_{n})$ with
$\length{\alpha_{1}} \geq 1$.  Now if $\bm{\alpha}$~is an address such
that both $\hat{B}_{d}(\bm{\alpha})$ and $B_{d}(\bm{\alpha})$ are
defined, then they evaluate to the same baker's map in the group~$nV$.
However, the former is not defined for addresses $\bm{\alpha} =
(\alpha_{1},\alpha_{2},\dots,\alpha_{n})$ with $\alpha_{d}$~empty.

\begin{lemma}
  \label{lem:A-map}
  Let $\bm{\alpha},\bm{\beta},\bm{\gamma},\bm{\delta} \in \Omega$ be
  addresses with $\bm{\alpha} \perp \bm{\beta}$ and $\bm{\gamma} \perp
  \bm{\delta}$.  Let $d,d'$~be indices in the range $2 \leq d,d' \leq
  n$.  Then
  \begin{enumerate}
  \item
    $A_{d}(\bm{\alpha},\bm{\beta})^{\swap{\bm{\gamma}}{\bm{\delta}}} =
    A_{d}\bigl(\bm{\alpha}\!\bullet\!\swap{\bm{\gamma}}{\bm{\delta}},
    \bm{\beta}\!\bullet\!\swap{\bm{\gamma}}{\bm{\delta}} \bigr)$,
    whenever both $\bm{\alpha} \bullet
    \swap{\bm{\gamma}}{\bm{\delta}}$ and $\bm{\beta} \bullet
    \swap{\bm{\gamma}}{\bm{\delta}}$ are defined;
  \item if every pair from $\{ \bm{\alpha}, \bm{\beta}, \bm{\gamma},
    \bm{\delta} \}$ are incomparable, then
    $A_{d}(\bm{\alpha},\bm{\beta})$ and
    $A_{d}(\bm{\gamma},\bm{\delta})$ commute;
  \item $A_{d}(\bm{\alpha},\bm{\beta}) =
    \swap{\bm{\alpha}.\bm{01}_{1}}{\bm{\alpha}.\bm{10}_{1}} \;
    A_{d}(\bm{\alpha}.\bm{0}_{1}, \bm{\beta}.\bm{0}_{1}) \;
    A_{d}(\bm{\alpha}.\bm{1}_{1}, \bm{\beta}.\bm{1}_{1}) \;
    \swap{\bm{\beta}.\bm{01}_{1}}{\bm{\beta}.\bm{10}_{1}}$;
  \item if $d' \neq d$, then $A_{d}(\bm{\alpha},\bm{\beta}) =
    A_{d}(\bm{\alpha}.\bm{0}_{d'},\bm{\beta}.\bm{0}_{d'}) \;
    A_{d}(\bm{\alpha}.\bm{1}_{d'},\bm{\beta}.\bm{1}_{d'})$.
  \end{enumerate}
\end{lemma}

\begin{proof}
  Part~(i) follows immediately by~\eqref{rel:conj}.  We deduce~(ii) by
  noting that Equation~\eqref{eq:A-def} expresses
  $A_{d}(\bm{\alpha},\bm{\beta})$ as a product of transpositions
  each of which, by~(i), commutes
  with~$A_{d}(\bm{\gamma},\bm{\delta})$.

  (iii)~We expand the right-hand side using the definition of the
  terms~$A_{d}$ and collect the transpositions using the conjugacy
  relations~\eqref{rel:conj}:
  \begin{align*}
    \lefteqn{%
      \swap{\bm{\alpha}.\bm{01}_{1}}{\bm{\alpha}.\bm{10}_{1}} \;
      A_{d}(\bm{\alpha}.\bm{0}_{1},\bm{\beta}.\bm{0}_{1}) \;
      A_{d}(\bm{\alpha}.\bm{1}_{1},\bm{\beta}.\bm{1}_{1}) \;
      \swap{\bm{\beta}.\bm{01}_{1}}{\bm{\beta}.\bm{10}_{1}}}
    \qquad\qquad\qquad \\*
    &=
    \swap{\bm{\alpha}.\bm{01}_{1}}{\bm{\alpha}.\bm{10}_{1}} \;
    \swap{\bm{\alpha}.\bm{00}_{1}}{\bm{\beta}.\bm{0}_{1}.\bm{0}_{d}}
    \;
    \swap{\bm{\alpha}.\bm{01}_{1}}{\bm{\beta}.\bm{0}_{1}.\bm{1}_{d}}
    \; \swap{\bm{\alpha}.\bm{0}_{1}}{\bm{\beta}.\bm{0}_{1}} \\*
    &\qquad \mbox{} \cdot
    \swap{\bm{\alpha}.\bm{10}_{1}}{\bm{\beta}.\bm{1}_{1}.\bm{0}_{d}}
    \;
    \swap{\bm{\alpha}.\bm{11}_{1}}{\bm{\beta}.\bm{1}_{1}.\bm{1}_{d}}
    \; \swap{\bm{\alpha}.\bm{1}_{1}}{\bm{\beta}.\bm{1}_{1}} \;
    \swap{\bm{\beta}.\bm{01}_{1}}{\bm{\beta}.\bm{10}_{1}} \\
    &=
    \swap{\bm{\alpha}.\bm{00}_{1}}{\bm{\beta}.\bm{0}_{1}.\bm{0}_{d}}
    \;\, \swap{\bm{\alpha}.\bm{01}_{1}}{\bm{\alpha}.\bm{10}_{1}}
    \;\,
    \swap{\bm{\alpha}.\bm{10}_{1}}{\bm{\beta}.\bm{1}_{1}.\bm{0}_{d}}
    \;\,
    \swap{\bm{\alpha}.\bm{01}_{1}}{\bm{\beta}.\bm{0}_{1}.\bm{1}_{d}}
    \\*
    &\qquad \mbox{} \cdot
    \swap{\bm{\alpha}.\bm{11}_{1}}{\bm{\beta}.\bm{1}_{1}.\bm{1}_{d}}
    \;\, \swap{\bm{\alpha}.\bm{0}_{1}}{\bm{\beta}.\bm{0}_{1}} \;\,
    \swap{\bm{\alpha}.\bm{1}_{1}}{\bm{\beta}.\bm{1}_{1}} \;\,
    \swap{\bm{\beta}.\bm{01}_{1}}{\bm{\beta}.\bm{10}_{1}} \\
    &=
    \swap{\bm{\alpha}.\bm{00}_{1}}{\bm{\beta}.\bm{0}_{1}.\bm{0}_{d}}
    \;\,
    \swap{\bm{\alpha}.\bm{01}_{1}}{\bm{\beta}.\bm{1}_{1}.\bm{0}_{d}}
    \;\,
    \swap{\bm{\alpha}.\bm{10}_{1}}{\bm{\beta}.\bm{0}_{1}.\bm{1}_{d}}
    \;\,
    \swap{\bm{\alpha}.\bm{11}_{1}}{\bm{\beta}.\bm{1}_{1}.\bm{1}_{d}}
    \\*
    &\qquad \mbox{} \cdot
    \swap{\bm{\alpha}.\bm{01}_{1}}{\bm{\alpha}.\bm{10}_{1}} \;\,
    \swap{\bm{\alpha}.\bm{0}_{1}}{\bm{\beta}.\bm{0}_{1}} \;\,
    \swap{\bm{\alpha}.\bm{1}_{1}}{\bm{\beta}.\bm{1}_{1}} \;\,
    \swap{\bm{\beta}.\bm{01}_{1}}{\bm{\beta}.\bm{10}_{1}} \\
    &= \swap{\bm{\alpha}.\bm{0}_{1}}{\bm{\beta}.\bm{0}_{d}} \;\,
    \swap{\bm{\alpha}.\bm{1}_{1}}{\bm{\beta}.\bm{1}_{d}} \;\,
    \swap{\bm{\alpha}.\bm{01}_{1}}{\bm{\alpha}.\bm{10}_{1}} \;\,
    \swap{\bm{\alpha}}{\bm{\beta}} \;\,
    \swap{\bm{\beta}.\bm{01}_{1}}{\bm{\beta}.\bm{10}_{1}}
    \qquad \text{(by~\eqref{rel:split})} \\
    &= \swap{\bm{\alpha}.\bm{0}_{1}}{\bm{\beta}.\bm{0}_{d}} \;\,
    \swap{\bm{\alpha}.\bm{1}_{1}}{\bm{\beta}.\bm{1}_{d}} \;\,
    \swap{\bm{\alpha}}{\bm{\beta}} \;\,
    \swap{\bm{\beta}.\bm{01}_{1}}{\bm{\beta}.\bm{10}_{1}}^{2}
    =
    A_{d}(\bm{\alpha},\bm{\beta}).
  \end{align*}

  (iv)~Apply the split relation~\eqref{rel:split} in the $d'$th
  coordinate to each transposition appearing
  in~$A_{d}(\bm{\alpha},\bm{\beta})$.  For example,
  $\swap{\bm{\alpha}.\bm{0}_{1}}{\bm{\beta}.\bm{0}_{d}} =
  \swap{\bm{\alpha}.\bm{0}_{d'}.\bm{0}_{1}}{\bm{\beta}.\bm{0}_{d'}.\bm{0}_{d}}
  \,
  \swap{\bm{\alpha}.\bm{1}_{d'}.\bm{0}_{1}}{\bm{\beta}.\bm{1}_{d'}.\bm{0}_{d}}$
  and similarly for the other transpositions.  Every pair of addresses
  from $\bm{\alpha}.\bm{0}_{d'}$,~$\bm{\alpha}.\bm{1}_{d'}$,
  $\bm{\beta}.\bm{0}_{d'}$ and~$\bm{\beta}.\bm{1}_{d'}$ are
  incomparable so rearranging the resulting formula
  for~$A_{d}(\bm{\alpha},\bm{\beta})$ yields the claimed result.
\end{proof}

\begin{lemma}
  \label{lem:prebaker}
  Let $d,d'$~be indices in the range $2 \leq d,d' \leq n$.
  \begin{enumerate}
  \item If $\bm{\alpha} = (\alpha_{1},\alpha_{2},\dots,\alpha_{n}) \in
    \Omega$ satisfies $\length{\alpha_{1}}, \length{\alpha_{d}} \geq
    1$ and if $\bm{\zeta},\bm{\eta} \in \Omega$ are such that every
    pair from $\{ \bm{\alpha}, \bm{\zeta}, \bm{\eta} \}$ are
    incomparable, then $\hat{B}_{d}(\bm{\alpha})$~commutes
    with~$\swap{\bm{\zeta}}{\bm{\eta}}$.
  \item If $\bm{\alpha} = \bm{\delta}.\bm{x}_{1}.\bm{y}_{d}$ for some
    $\bm{\delta} \in \Omega$ and some $x,y \in \{0,1\}$, then
    \begin{equation}
      \hat{B}_{d}(\bm{\alpha}) =
      A_{d}(\bm{\delta}.\bm{x}_{1},\bm{\gamma}) \;
      \swap{\bm{\delta}.\bm{x}_{1}.\bm{01}_{d}}{\bm{\delta}.\bm{x}_{1}.\bm{10}_{d}}
      \; A_{d}(\bm{\gamma},\bm{\delta}.\bm{x}_{1}.\bar{\bm{y}}_{d})
      \label{eq:hatB-form}
    \end{equation}
    for any address $\bm{\gamma} =
    (\gamma_{1},\gamma_{2},\dots,\gamma_{n})$ such that $\bm{\gamma}
    \perp \bm{\delta}.\bm{x}_{1}$ and\/ $\length{\gamma_{1}},
    \length{\gamma_{d}} \geq 1$.
  \item If $y \in \{0,1\}$, then
    $\hat{B}_{d}(\bm{0}_{1}.\bm{y}_{d})^{\swap{\bm{0}_{1}}{\bm{1}_{1}}}
    = \hat{B}_{d}(\bm{1}_{1}.\bm{y}_{d})$.
  \item If $y \in \{0,1\}$, then
    $\hat{B}_{d}(\bm{0}_{1}.\bm{y}_{d})^{\swap{\bm{0}_{1}}{\bm{10}_{1}}}
    = \hat{B}_{d}(\bm{10}_{1}.\bm{y}_{d})$ and
    $\hat{B}_{d}(\bm{0}_{1}.\bm{y}_{d})^{\swap{\bm{0}_{1}}{\bm{11}_{1}}}
    = \hat{B}_{d}(\bm{11}_{1}.\bm{y}_{d})$.
  \item If $x \in \{0,1\}$, then
    $\hat{B}_{d}(\bm{x0}_{1}.\bm{1}_{d})$ and
    $\hat{B}_{d}(\bm{x1}_{1}.\bm{0}_{d})$ commute.
  \item If $x,y \in \{0,1\}$, then $\hat{B}_{d}(\bm{x}_{1}.\bm{y}_{d})
    = \swap{\bm{x01}_{1}.\bm{y}_{d}}{\bm{x10}_{1}.\bm{y}_{d}} \;
    \hat{B}_{d}(\bm{x0}_{1}.\bm{y}_{d}) \;
    \hat{B}_{d}(\bm{x1}_{1}.\bm{y}_{d})$.
  \item If $d' \neq d$, then $\hat{B}_{d}(\bm{x}_{1}.\bm{y}_{d}) =
    \hat{B}_{d}(\bm{x}_{1}.\bm{y}_{d}.\bm{0}_{d'}) \;
    \hat{B}_{d}(\bm{x}_{1}.\bm{y}_{d}.\bm{1}_{d'})$.
  \end{enumerate}
\end{lemma}

If we take $\bm{\delta} = \bm{\epsilon}$ in Part~(ii) then it tells us
that the definitions in Equation~\eqref{eq:hatB-def} are independent
of the choice of address~$\bm{\gamma}$ used.  Consequently,
$\hat{B}_{d}(\bm{\alpha})$~does depend only on the
address~$\bm{\alpha}$.

\begin{proof}
  (i)~Repeatedly apply the split relation~\eqref{rel:split} to
  express~$\swap{\bm{\zeta}}{\bm{\eta}}$ as a product of
  transpositions~$\swap{\bm{\zeta}'}{\bm{\eta}'}$ having entries with
  sufficiently long coordinates that the values
  $\bm{\zeta}'\bullet\hat{B}_{d}(\bm{\alpha})$ and
  $\bm{\eta}'\bullet\hat{B}_{d}(\bm{\alpha})$ are defined.  These
  values therefore coincide with those obtained when the corresponding
  baker's map in~$nV$ is applied to $\bm{\zeta}'$~and~$\bm{\eta}'$.
  Since $\bm{\zeta}$~and~$\bm{\eta}$ are both incomparable
  with~$\bm{\alpha}$, we conclude
  $\bm{\zeta}'\bullet\hat{B}_{d}(\bm{\alpha}) = \bm{\zeta}'$ and
  $\bm{\eta}'\bullet\hat{B}_{d}(\bm{\alpha}) = \bm{\eta}'$.  It then
  follows, by the conjugacy relation~\eqref{rel:conj}, that
  $\hat{B}_{d}(\bm{\alpha})$~commutes with all
  such~$\swap{\bm{\zeta}'}{\bm{\eta}'}$ and hence also with their
  product~$\swap{\bm{\zeta}}{\bm{\eta}}$.

  (ii)~One formula of the form given in~\eqref{eq:hatB-form}
  for~$\hat{B}_{d}(\bm{\delta}.\bm{x}_{1}.\bm{y}_{d})$ is obtained by
  inserting~$\bm{\delta}$ as prefix into the entries of the
  transpositions in the definition
  of~$\hat{B}_{d}(\bm{x}_{1}.\bm{y}_{d})$ in
  Equation~\eqref{eq:hatB-def}.  Start with one such valid
  formula~\eqref{eq:hatB-form} involving a particular
  address~$\bm{\gamma}$ and consider another potential
  address~$\bm{\gamma}' = (\gamma_{1}',\gamma_{2}',\dots,\gamma_{n}')$
  satisfying $\bm{\gamma}' \perp \bm{\delta}.\bm{x}_{1}$ and
  $\length{\gamma_{d}'} \geq 1$.  If $\bm{\gamma}' \perp \bm{\gamma}$,
  conjugate by~$\swap{\bm{\gamma}}{\bm{\gamma}'}$ and use part~(i) to
  produce the required formula
  for~$\hat{B}_{d}(\bm{\delta}.\bm{x}_{1}.\bm{y}_{d})$.  If
  $\bm{\gamma}$~and~$\bm{\gamma}'$ are not incomparable, then as
  $\length{\gamma_{d}}, \length{\gamma'_{d}} \geq 1$ there is another
  address~$\bm{\zeta}$, with non-empty $d$th coordinate, that is
  incomparable with all three of $\bm{\gamma}$,~$\bm{\gamma}'$
  and~$\bm{\delta}.\bm{x}_{1}$.  Now conjugate by the product
  $\swap{\bm{\gamma}}{\bm{\zeta}} \; \swap{\bm{\gamma}'}{\bm{\zeta}}$,
  again using part~(i), to produce the required
  formula~\eqref{eq:hatB-form} involving the address~$\bm{\gamma}'$.

  Part~(iii) follows immediately from the definition of
  the~$\hat{B}_{d}$ elements.

  (iv)~We establish here the first equation in the case when $y = 0$.
  The other formulae are similar.  First use part~(ii) to assume that
  in the expression~\eqref{eq:hatB-def}
  for~$\hat{B}_{d}(\bm{0}_{1}.\bm{0}_{d})$ the address $\bm{\gamma} =
  (\gamma_{1},\gamma_{2},\dots,\gamma_{n})$ satisfies $11 \preceq
  \gamma_{1}$.  Consequently $\bm{\gamma} \bullet
  \swap{\bm{0}_{1}}{\bm{10}_{1}} = \bm{\gamma}$ and $\bm{\gamma} =
  \bm{1}_{1}.\bm{\delta}$ where $\bm{\delta} \perp \bm{0}_{1}$.
  Hence, using Lemma~\ref{lem:A-map}(i),
  \begin{multline*}
    \hat{B}_{d}(\bm{0}_{1}.\bm{0}_{d})^{\swap{\bm{0}_{1}}{\bm{10}_{1}}}
    = A_{d}(\bm{10}_{1},\bm{\gamma}) \;
    \swap{\bm{10}_{1}.\bm{01}_{d}}{\bm{10}_{1}.\bm{10}_{d}} \;
    A_{d}(\bm{\gamma},\bm{10}_{1}.\bm{1}_{d}) \\
    = \bm{1}_{1}.\bigl( A_{d}(\bm{0}_{1},\bm{\delta}) \;
    \swap{\bm{0}_{1}.\bm{01}_{d}}{\bm{0}_{1}.\bm{10}_{d}} \;
    A_{d}(\bm{\delta},\bm{0}_{1}.\bm{1}_{d}) \bigr)
    = \bm{1}_{1}.\hat{B}_{d}(\bm{0}_{1}.\bm{0}_{d})
    = \hat{B}_{d}(\bm{10}_{1}.\bm{0}_{d}).
  \end{multline*}

  (v)~Using part~(ii), we can assume that
  \begin{align*}
    \hat{B}_{d}(\bm{x0}_{1}.\bm{1}_{d}) &=
    A_{d}(\bm{x0}_{1},\bm{\gamma}) \;
    \swap{\bm{x0}_{1}.\bm{01}_{d}}{\bm{x0}_{1}.\bm{10}_{d}} \;
    A_{d}(\bm{\gamma},\bm{x0}_{1}.\bm{0}_{d}) \\
    \hat{B}_{d}(\bm{x1}_{1}.\bm{0}_{d}) &=
    A_{d}(\bm{x1}_{1},\bm{\gamma}') \;
    \swap{\bm{x1}_{1}.\bm{01}_{d}}{\bm{x1}_{1}.\bm{10}_{d}} \;
    A_{d}(\bm{\gamma}',\bm{x1}_{1}.\bm{1}_{d})
  \end{align*}
  where $\bm{\gamma}$~and~$\bm{\gamma}'$ are addresses with
  $\bm{\gamma} \perp \bm{\gamma}'$ and $\bm{x}_{1} \perp
  \bm{\gamma},\bm{\gamma}'$.  These commute by
  Lemma~\ref{lem:A-map}(i)--(ii).

  (vi)~We present the case $x = y = 0$, with the other cases
  established by similar calculations.  Recall that $\bm{\gamma}
  \perp \bm{0}_{1}$ in our definition~\eqref{eq:hatB-def}
  of~$\hat{B}_{d}(\bm{0}_{1}.\bm{0}_{d})$.  In the following
  calculation, we begin by applying Lemma~\ref{lem:A-map}(iii) to the
  terms appearing in the definition
  of~$\hat{B}_{d}(\bm{0}_{1}.\bm{0}_{d})$:
  \begin{align*}
    \hat{B}_{d}(\bm{0}_{1}.\bm{0}_{d}) &=
    \swap{\bm{001}_{1}}{\bm{010}_{1}} \;
    A_{d}(\bm{00}_{1},\bm{\gamma}.\bm{0}_{1}) \;
    A_{d}(\bm{01}_{1},\bm{\gamma}.\bm{1}_{1}) \;
    \swap{\bm{\gamma}.\bm{01}_{1}}{\bm{\gamma}.\bm{10}_{1}} \;
    \swap{\bm{0}_{1}.\bm{01}_{d}}{\bm{0}_{1}.\bm{10}_{d}} \\*
    &\qquad \mbox{} \cdot
    \swap{\bm{\gamma}.\bm{01}_{1}}{\bm{\gamma}.\bm{10}_{1}} \;
    A_{d}(\bm{\gamma}.\bm{0}_{1},\bm{00}_{1}.\bm{1}_{d}) \;
    A_{d}(\bm{\gamma}.\bm{1}_{1},\bm{01}_{1}.\bm{1}_{d}) \;
    \swap{\bm{001}_{1}.\bm{1}_{d}}{\bm{010}_{1}.\bm{1}_{d}} \\
    &= \swap{\bm{001}_{1}}{\bm{010}_{1}} \;
    A_{d}(\bm{00}_{1},\bm{\gamma}.\bm{0}_{1}) \;
    A_{d}(\bm{01}_{1},\bm{\gamma}.\bm{1}_{1}) \;
    \swap{\bm{00}_{1}.\bm{01}_{d}}{\bm{00}_{1}.\bm{10}_{d}} \\*
    &\qquad \mbox{} \cdot
    \swap{\bm{01}_{1}.\bm{01}_{d}}{\bm{01}_{1}.\bm{10}_{d}} \;
    A_{d}(\bm{\gamma}.\bm{0}_{1},\bm{00}_{1}.\bm{1}_{d}) \;
    A_{d}(\bm{\gamma}.\bm{1}_{1},\bm{01}_{1}.\bm{1}_{d}) \;
    \swap{\bm{001}_{1}.\bm{1}_{d}}{\bm{010}_{1}.\bm{1}_{d}} \\
    &= \swap{\bm{001}_{1}}{\bm{010}_{1}} \;
    A_{d}(\bm{00}_{1},\bm{\gamma}.\bm{0}_{1}) \;
    \swap{\bm{00}_{1}.\bm{01}_{d}}{\bm{00}_{1}.\bm{10}_{d}} \;
    A_{d}(\bm{\gamma}.\bm{0}_{1},\bm{00}_{1}.\bm{1}_{d}) \\*
    &\qquad \mbox{} \cdot A_{d}(\bm{01}_{1},\bm{\gamma}.\bm{1}_{1})
    \; \swap{\bm{01}_{1}.\bm{01}_{d}}{\bm{01}_{1}.\bm{10}_{d}} \;
    A_{d}(\bm{\gamma}.\bm{1}_{1},\bm{01}_{1}.\bm{1}_{d}) \;
    \swap{\bm{001}_{1}.\bm{1}_{d}}{\bm{010}_{1}.\bm{1}_{d}} \\
    &= \swap{\bm{001}_{1}}{\bm{010}_{1}} \;
    \hat{B}_{d}(\bm{00}_{1}.\bm{0}_{d}) \;
    \hat{B}_{d}(\bm{01}_{1}.\bm{0}_{d}) \;
    \swap{\bm{001}_{1}.\bm{1}_{d}}{\bm{010}_{1}.\bm{1}_{d}}
    \qquad \text{(using part~(ii))} \\
    &= \swap{\bm{001}_{1}.\bm{0}_{d}}{\bm{010}_{1}.\bm{0}_{d}} \;
    \hat{B}_{d}(\bm{00}_{1}.\bm{0}_{d}) \;
    \hat{B}_{d}(\bm{01}_{1}.\bm{0}_{d}).
  \end{align*}

  (vii)~This follows by applying Lemma~\ref{lem:A-map}(iv) and the
  split relation~\eqref{rel:split} to terms appearing in the
  formula~\eqref{eq:hatB-def}, and then rearranging in a similar way
  to the proof of Lemma~\ref{lem:A-map}(iv).
\end{proof}

\begin{lemma}
  \label{lem:baker}
  Let $d,d'$~be indices in the range $2 \leq d,d' \leq n$.
  \begin{enumerate}
  \item If $\bm{\alpha} = (\alpha_{1},\alpha_{2},\dots,\alpha_{n}) \in
    \Omega$ satisfies $\length{\alpha_{1}} \geq 1$ and if $\bm{\zeta},
    \bm{\eta} \in \Omega$ are such that every pair from
    $\{ \bm{\alpha}, \bm{\zeta}, \bm{\eta} \}$ are incomparable, then
    $B_{d}(\bm{\alpha})$~commutes
    with~$\swap{\bm{\zeta}}{\bm{\eta}}$.  In particular, for any $x
    \in \{0,1\}$, the element~$B_{d}(\bm{x}_{1})$ commutes with any
    element of the form~$\bar{\bm{x}}_{1}.g$.
  \item $B_{d}(\bm{0}_{1})^{\swap{\bm{0}_{1}}{\bm{1}_{1}}} =
    B_{d}(\bm{1}_{1})$.
  \item $B_{d}(\bm{0}_{1})^{\swap{\bm{0}_{1}}{\bm{10}_{1}}} =
    B_{d}(\bm{10}_{1})$ and
    $B_{d}(\bm{0}_{1})^{\swap{\bm{0}_{1}}{\bm{11}_{1}}} = B_{d}(\bm{11}_{1})$.
  \item If $x \in \{0,1\}$, then $B_{d}(\bm{x}_{1}) =
    \swap{\bm{x01}_{1}}{\bm{x10}_{1}} \; B_{d}(\bm{x0}_{1}) \;
    B_{d}(\bm{x1}_{1})$.
  \item $B_{d}(\bm{0}_{1})$~and~$B_{d'}(\bm{1}_{1})$ commute.
  \item If $d' \neq d$, then $B_{d}(\bm{x}_{1}) =
    B_{d}(\bm{x}_{1}.\bm{0}_{d'}) \;
    B_{d}(\bm{x}_{1}.\bm{1}_{d'})$.
  \end{enumerate}
\end{lemma}

\begin{proof}
  The first part of~(i) is established by the same argument as used in
  Lemma~\ref{lem:prebaker}(i) and the remainder follows immediately.
  Parts (ii)~and~(iii) follow using parts
  (iii)~and~(iv), respectively, of Lemma~\ref{lem:prebaker}, while
  part~(vi) is an extension of Lemma~\ref{lem:prebaker}(vii) that is
  established similarly.

  We establish part~(iv) in the case that $x = 0$.  First apply
  Lemma~\ref{lem:prebaker}(vi) to the two $\hat{B}_{d}$~terms in the
  definition~\eqref{eq:B-def} and then use parts (i)~and~(v) of
  Lemma~\ref{lem:prebaker} to rearrange terms:
  \begin{align*}
    B_{d}(\bm{0}_{1}) &=
    \swap{\bm{001}_{1}.\bm{0}_{d}}{\bm{010}_{1}.\bm{0}_{d}} \;
    \hat{B}_{d}(\bm{00}_{1}.\bm{0}_{d}) \;
    \hat{B}_{d}(\bm{01}_{1}.\bm{0}_{d}) \;
    \swap{\bm{001}_{1}.\bm{1}_{d}}{\bm{010}_{1}.\bm{1}_{d}} \\*
    &\qquad \mbox{} \cdot \hat{B}_{d}(\bm{00}_{1}.\bm{1}_{d}) \;
    \hat{B}_{d}(\bm{01}_{1}.\bm{1}_{d}) \;
    \swap{\bm{0}_{1}.\bm{01}_{d}}{\bm{0}_{1}.\bm{10}_{d}} \\
    &= \swap{\bm{001}_{1}.\bm{0}_{d}}{\bm{010}_{1}.\bm{0}_{d}} \;
    \swap{\bm{001}_{1}.\bm{1}_{d}}{\bm{010}_{1}.\bm{1}_{d}} \;
    \hat{B}_{d}(\bm{00}_{1}.\bm{0}_{d}) \;
    \hat{B}_{d}(\bm{00}_{1}.\bm{1}_{d}) \\*
    &\qquad \mbox{} \cdot \hat{B}_{d}(\bm{01}_{1}.\bm{0}_{d}) \;
    \hat{B}_{d}(\bm{01}_{1}.\bm{1}_{d}) \;
    \swap{\bm{0}_{1}.\bm{01}_{d}}{\bm{0}_{1}.\bm{10}_{d}} \\
    &= \swap{\bm{001}_{1}}{\bm{010}_{1}} \;
    \hat{B}_{d}(\bm{00}_{1}.\bm{0}_{d}) \;
    \hat{B}_{d}(\bm{00}_{1}.\bm{1}_{d}) \;
    \hat{B}_{d}(\bm{01}_{1}.\bm{0}_{d}) \;
    \hat{B}_{d}(\bm{01}_{1}.\bm{1}_{d}) \\*
    &\qquad \mbox{} \cdot
    \swap{\bm{00}_{1}.\bm{01}_{d}}{\bm{00}_{1}.\bm{10}_{d}} \;
    \swap{\bm{01}_{1}.\bm{01}_{d}}{\bm{01}_{1}.\bm{10}_{d}} \\
    &= \swap{\bm{001}_{1}}{\bm{010}_{1}} \;
    \hat{B}_{d}(\bm{00}_{1}.\bm{0}_{d}) \;
    \hat{B}_{d}(\bm{00}_{1}.\bm{1}_{d}) \;
    \swap{\bm{00}_{1}.\bm{01}_{d}}{\bm{00}_{1}.\bm{10}_{d}} \\*
    &\qquad \mbox{} \cdot \hat{B}_{d}(\bm{01}_{1}.\bm{0}_{d}) \;
    \hat{B}_{d}(\bm{01}_{1}.\bm{1}_{d}) \;
    \swap{\bm{01}_{1}.\bm{01}_{d}}{\bm{01}_{1}.\bm{10}_{d}} \\
    &= \swap{\bm{001}_{1}}{\bm{010}_{1}} \; B_{d}(\bm{00}_{1}) \;
    B_{d}(\bm{01}_{1}).
  \end{align*}

  Finally, it then follows that $B_{d}(\bm{0}_{1}) = \bm{0}_{1}.\bigl(
  \swap{\bm{01}_{1}}{\bm{10}_{1}} \; B_{d}(\bm{0}_{1}) \;
  B_{d}(\bm{1}) \bigr)$ commutes with~$B_{d'}(\bm{1}_{1})$ using
  part~(i).  Hence we have established the remaining part~(v).
\end{proof}

We now consider one of the presentations for Brin's higher-dimensional
Thompson group~$nV$ given by Hennig--Matucci~\cite{HM}.  They define
generators $X_{m,d}$ (for $m \geq 0$ and $1 \leq d \leq n$),
\ $C_{m,d}$ (for $m \geq 0$ and $2 \leq d \leq n$), \ $\pi_{m}$ (for
$m \geq 0$) and $\bar{\pi}_{m}$ (for $m \geq 0$) and describe eighteen
families of relations (numbered (1)--(18) on pages~59--60
of~\cite{HM}).  They observe, in~\cite[Theorem~23]{HM}, that these do
indeed give a presentation for~$nV$.  Since we write our maps on the
right, we shall convert the relations to our setting by reversing each
one and record these now for reference.  We have also changed some of
the labels on Hennig--Matucci's generators appearing in the lists so
that our arguments can be unified when establishing
Proposition~\ref{prop:HM-rels} below.  In the
families~\eqref{eq:cut1}--\eqref{eq:cut7}, one should assume that $1
\leq d,d' \leq n$:
\begin{HMeqs}
  X_{m,d} \, X_{q,d'} &= X_{q+1,d'} \, X_{m,d} &&\text{for $m <
    q$,} \label{eq:cut1} \\
  X_{m,d} \, \pi_{q} &= \pi_{q+1} \, X_{m,d} &&\text{for $m <
    q$,} \label{eq:cut2} \\
  X_{m,d} \, \pi_{m} &= \pi_{m+1} \, \pi_{m} \, X_{m+1,d} &&\text{for
    $m \geq 0$,} \label{eq:cut3} \\
  X_{q,d} \, \pi_{m} &= \pi_{m} \, X_{q,d} &&\text{for $q > m+1$,}
  \label{eq:cut4} \\
  X_{m,d} \, \bar{\pi}_{q} &= \bar{\pi}_{q+1} \, X_{m,d} &&\text{for
    $m < q$,} \label{eq:cut5} \\
  X_{m,1} \, \bar{\pi}_{m} &= \bar{\pi}_{m+1} \, \pi_{m} &&\text{for
    $m \geq 0$,} \label{eq:cut6} \\
  X_{m,d'} \, X_{m+1,d'} \, X_{m,d} &= \pi_{m+1} \, X_{m,d} \,
  X_{m+1,d} \, X_{m,d'} &&\text{for $m \geq 0$ and $d \neq
    d'$}. \label{eq:cut7} \\
\intertext{The second collection of relations is as below.  Note that
  we have adjusted the range of the parameters in~\eqref{eq:perm8} to
  bring it into line with the relations given by Brin
  (see~\cite[Eqn.~(22)]{Brin-presentations}).}
  \pi_{m} \, \pi_{q} &= \pi_{q} \, \pi_{m} &&\text{for $\modulus{m -
      q} \geq 2$,} \label{eq:perm8} \\
  \pi_{m} \, \pi_{m+1} \, \pi_{m} &= \pi_{m+1} \, \pi_{m} \, \pi_{m+1}
  &&\text{for $m \geq 0$,} \\
  \pi_{m} \, \bar{\pi}_{q} &= \bar{\pi}_{q} \, \pi_{m} &&\text{for $q
    \geq m+2$,} \\
  \pi_{m} \, \bar{\pi}_{m+1} \, \pi_{m} &= \bar{\pi}_{m+1} \, \pi_{m}
  \, \bar{\pi}_{m+1} &&\text{for $m \geq 0$,} \\
  \pi_{m}^{2} &= 1 &&\text{for $m \geq 0$,} \\
  \bar{\pi}_{m}^{2} &= 1 &&\text{for $m \geq 0$}. \label{eq:perm13}
  \\
\intertext{Finally, in the
  families~\eqref{eq:bakers14}--\eqref{eq:bakers18}, $2 \leq d \leq n$
  and $1 \leq d' \leq n$, unless otherwise indicated:}
  X_{m,d} \, \bar{\pi}_{m} &= \bar{\pi}_{m+1} \, \pi_{m} \, C_{m+1,d}
  &&\text{for $m \geq 0$,} \label{eq:bakers14} \\
  X_{m,d'} \, C_{q,d} &= C_{q+1,d} \, X_{m,d'} &&\text{for $m < q$,}
  \label{eq:bakers15} \\
  X_{m,1} \, C_{m,d} &= \pi_{m+1} \, C_{m+2,d} \, X_{m,d} &&\text{for
    $m \geq 0$,} \label{eq:bakers16} \\
  C_{q,d} \, \pi_{m} &= \pi_{m} \, C_{q,d} &&\text{for $q > m + 1$,}
  \label{eq:bakers17} \\
  C_{m+2,d'} \, X_{m,d'} \, C_{m,d} &= \pi_{m+1} \, C_{m+2,d} \,
  X_{m,d} \, C_{m,d'} &&\text{for $m \geq 0$, \ $1 < d' < d \leq
    n$}. \label{eq:bakers18}
\end{HMeqs}

We now define the elements of our group~$G_{\infty}$ that correspond
to the above generators.  In the following $d$~is an index with $2
\leq d \leq n$.  First we set
\begin{equation}
  \begin{aligned}
    X_{0,1} &= \swap{\bm{0}_{1}}{\bm{1}_{1}} \;\,
    \swap{\bm{0}_{1}}{\bm{10}_{1}} \;\,
    \swap{\bm{10}_{1}}{\bm{11}_{1}},
    \qquad\qquad
    &X_{0,d} &= X_{0,1} \: B_{d}(\bm{1}_{1}), \\*
    \pi_{0} &= \swap{\bm{01}_{1}}{\bm{1}_{1}},
    &\bar{\pi}_{0} &= \swap{\bm{0}_{1}}{\bm{1}_{1}}, \\*
    C_{0,d} &= \swap{\bm{01}_{1}}{\bm{10}_{1}} \;\, B_{d}(\bm{0}_{1})
    \;\, B_{d}(\bm{1}_{1}),
    &C_{1,d} &= B_{d}(\bm{0}_{1}).
  \end{aligned}
  \label{eq:HMgens1}
\end{equation}
These are extended, for positive integers $m \geq 1$, to
\begin{equation}
  \begin{aligned}
  X_{m,d} &= \bm{0}_{1}^{m}.X_{0,d} &&\text{for $1 \leq d \leq n$},
  \\*
  \pi_{m} &= \bm{0}_{1}^{m}.\pi_{0}, \\*
  \bar{\pi}_{m} &= \bm{0}_{1}^{m}.\bar{\pi}_{0}, \\*
  C_{m,d} &= \bm{0}_{1}^{m-1}.C_{1,d} &&\text{for $m \geq 2$ and $2
    \leq d \leq n$},
  \end{aligned}
  \label{eq:HMgens2}
\end{equation}
where $\bm{0}_{1}^{m}$~denotes $\bm{00}\dots{}\bm{0}_{1} =
(00\dots0,\epsilon,\dots,\epsilon)$ with the word~$00\dots0$ having
length~$m$.

Note that, by Lemma~\ref{lem:baker}(iv), \ $C_{1,d} =
\bm{0}_{1}.C_{0,d}$.  Consequently, the definition of~$C_{m,d}$ can be
extended to include the case~$ m =1$; that is, $C_{m,d} =
\bm{0}_{1}^{m}.C_{0,d}$ for all $m \geq 1$ and all indices~$d$ in the
range $2 \leq d \leq n$.  This will enable us to treat these baker's
maps in a uniform manner.

\begin{lemma}
  \label{lem:fullbaker}
  Let $d$~and~$d'$ be indices in the range $2 \leq d,d' \leq n$.
  \begin{enumerate}
  \item $B_{d}(\bm{0}_{1})^{C_{0,d'}} = \bm{0}_{d'}.C_{0,d}$ and
    $B_{d}(\bm{1}_{1})^{C_{0,d'}} = \bm{1}_{d'}.C_{0,d}$.
  \item If $d \neq d'$, then $C_{0,d} = (\bm{0}_{d'}.C_{0,d}) \:
    (\bm{1}_{d'}.C_{0,d})$.
  \end{enumerate}
\end{lemma}

By inserting the prefix~$\bm{1}_{1}$ into the entries of
transpositions in part~(i) of this lemma, it follows with use of
Lemma~\ref{lem:baker}(iv) that:
\begin{equation}
  B_{d}(\bm{10}_{1})^{B_{d'}(\bm{1}_{1})} =
  B_{d}(\bm{1}_{1}.\bm{0}_{d'})
  \qquad \text{and} \qquad
  B_{d}(\bm{11}_{1})^{B_{d'}(\bm{1}_{1})} =
  B_{d}(\bm{1}_{1}.\bm{1}_{d'})
  \label{eq:baker-conj}
\end{equation}

\begin{proof}
  (i)~First note that, by suitable choice of~$\bm{\gamma}$ appearing
  in the definition~\eqref{eq:hatB-def}, we can express~$C_{0,d'}$ as
  a product of transpositions whose entries have non-empty coordinates
  only for index~$1$ and index~$d'$.  Consequently, provided the
  index~$1$ and index~$d'$ coordinates of an address~$\bm{\zeta}$ are
  sufficiently long, $\bm{\zeta} \bullet C_{0,d'}$ is defined.
  Furthermore $C_{0,d'}$~evaluates in~$nV$ to the (primary) baker's
  map with full support on~$\Gamma$, so this value $\bm{\zeta} \bullet
  C_{0,d'}$ coincides with that obtained when we act with the baker's
  map.  Now if $x \in \{0,1\}$, then $(\bm{x}_{1}.\bm{\zeta}) \bullet
  C_{0,d'}$ is also defined (as the required coordinates are
  sufficiently long) and equals $\bm{x}_{d'}.\bm{\zeta}$ in view of
  how the baker's map acts.

  If $\swap{\bm{\alpha}}{\bm{\beta}}$~is any transposition,
  apply~\eqref{rel:split} repeatedly to express it as a product of
  transpositions~$\swap{\bm{\zeta}}{\bm{\eta}}$ with the index~$1$ and
  index~$d$ coordinates of $\bm{\zeta}$~and~$\bm{\eta}$ sufficiently
  long that the partial action of~$C_{0,d'}$ upon them is defined.  By
  inserting~$\bm{x}_{1}$ as prefix into all the transpositions
  involved, we express
  $\swap{\bm{x}_{1}.\bm{\alpha}}{\bm{x}_{1}.\bm{\beta}}$ as a product
  of
  transpositions~$\swap{\bm{x}_{1}.\bm{\zeta}}{\bm{x_{1}}.\bm{\eta}}$.
  Since $\bm{x}_{1}.\bm{\zeta} \bullet C_{0,d'} =
  \bm{x}_{d'}.\bm{\zeta}$ by our above observation and similarly for
  $\bm{x}_{1}.\bm{\eta}$, we deduce
  \[
  \swap{\bm{x}_{1}.\bm{\zeta}}{\bm{x}_{1}.\bm{\eta}}^{C_{0,d'}} =
  \swap{\bm{x}_{d'}.\bm{\zeta}}{\bm{x}_{d'}.\bm{\eta}}.
  \]
  We now recombine the resulting transpositions
  using~\eqref{rel:split} to conclude
  \[
  \swap{\bm{x}_{1}.\bm{\alpha}}{\bm{x}_{1}.\bm{\beta}}^{C_{0,d'}} =
  \swap{\bm{x}_{d'}.\bm{\alpha}}{\bm{x}_{d'}.\bm{\beta}}.
  \]

  Now if $g$~is any element of~$G_{\infty}$, then it is a product of
  transpositions and it follows from the above calculation that
  \[
  (\bm{x}_{1}.g)^{C_{0,d'}} = \bm{x}_{d'}.g
  \]
  for any $x \in \{0,1\}$.  The claimed equations now follow by taking
  $g = C_{0,d}$ and noting, by use Lemma~\ref{lem:baker}(iv), that
  $B_{d}(\bm{x}_{1}) = \bm{x}_{1}.C_{0,d}$.

  Part~(ii) is an extension of Lemma~\ref{lem:baker}(vi) established
  by a similar argument.
\end{proof}

\begin{prop}
  \label{prop:HM-rels}
  The elements $X_{m,d}$,~$C_{m,d}$, $\pi_{m}$ and~$\bar{\pi}_{m}$
  of~$G_{\infty}$ defined in Equations~\eqref{eq:HMgens1}
  and~\eqref{eq:HMgens2} satisfy the
  relations~\eqref{eq:cut1}--\eqref{eq:bakers18}.
\end{prop}

Establishing this proposition completes the proof of
Theorem~\ref{thm:infpres} since it establishes that the surjective
homomorphism $\phi \colon G_{\infty} \to nV$ has trivial kernel.

\begin{proof}
  First, as observed earlier, if $u = v$ is a relation in~$G_{\infty}$
  then so is $\bm{0}_{1}.u = \bm{0}_{1}.v$.  Consequently, it suffices
  to establish each of~\eqref{eq:cut1}--\eqref{eq:bakers18} only when $m = 0$.
  Relations~\eqref{eq:perm8}--\eqref{eq:perm13} are the most
  straightforward to verify and follow directly from the assumed
  relations involving transpositions (i.e., \eqref{rel:order}
  and~\eqref{rel:conj}).  Relation~\eqref{eq:cut6} is established in a
  similar manner.  When $q > 1$, observe that both
  $X_{q,d}$~and~$C_{q,d}$ is a product of transpositions all of whose
  entries have $\bm{00}_{1}$ as a prefix.  These transpositions are
  therefore disjoint from~$\pi_{0}$ and relations
  \eqref{eq:cut4}~and~\eqref{eq:bakers17} follow.  We now describe the
  details involved in the other relations.

  Note that $\bm{00}_{1} \bullet X_{0,1} = \bm{0}_{1}$.  Hence, for
  any $g \in G_{\infty}$,
  \[
  X_{0,1}^{-1} (\bm{00}_{1}.g) X_{0,1} = \bm{0}_{1}.g.
  \]
  This establishes~\eqref{eq:cut1}, \eqref{eq:cut2}
  and~\eqref{eq:cut5} in the case when $m = 0$ and $d = 1$, and it
  also establishes~\eqref{eq:bakers15} in the case when $m = 0$ and
  $d' = 1$.  We extend the first three to $d \geq 2$ by use of
  Lemma~\ref{lem:baker}(i) to tell us that
  $B_{d}(\bm{1}_{1})$~commutes with each of $X_{q,d'}$, $\pi_{q}$
  and~$\bar{\pi}_{q}$ for $q > 0$.  Similarly, we
  extend~\eqref{eq:bakers15} to the case when $d' \geq 2$ by using the
  same fact to show $B_{d'}(\bm{1}_{1})$~commutes with~$C_{q,d}$ for
  $q \geq 2$ and, by Lemma~\ref{lem:baker}(v), also with~$C_{1,d}$.

  The relation~\eqref{eq:cut3} when $m = 0$ and $d = 1$ is established
  by collecting transpositions using~\eqref{rel:conj} and one
  application of~\eqref{rel:split}:
  \begin{align*}
    \pi_{1} \, \pi_{0} \, X_{1,1} &= \swap{\bm{001}_{1}}{\bm{01}_{1}}
    \; \swap{\bm{01}_{1}}{\bm{1}_{1}} \;
    \swap{\bm{00}_{1}}{\bm{01}_{1}} \;
    \swap{\bm{00}_{1}}{\bm{010}_{1}} \;
    \swap{\bm{010}_{1}}{\bm{011}_{1}} \\
    &= \swap{\bm{001}_{1}}{\bm{01}_{1}} \;
    \swap{\bm{00}_{1}}{\bm{1}_{1}} \;
    \swap{\bm{00}_{1}}{\bm{10}_{1}} \;
    \swap{\bm{10}_{1}}{\bm{11}_{1}} \;
    \swap{\bm{01}_{1}}{\bm{1}_{1}} \\
    &= \swap{\bm{00}_{1}}{\bm{1}_{1}} \;
    \swap{\bm{01}_{1}}{\bm{11}_{1}} \;
    \swap{\bm{00}_{1}}{\bm{10}_{1}} \;
    \swap{\bm{10}_{1}}{\bm{11}_{1}} \;
    \swap{\bm{01}_{1}}{\bm{1}_{1}} \\
    &= \swap{\bm{00}_{1}}{\bm{1}_{1}} \;
    \swap{\bm{0}_{1}}{\bm{1}_{1}} \;
    \swap{\bm{10}_{1}}{\bm{11}_{1}} \;
    \swap{\bm{01}_{1}}{\bm{1}_{1}} \\
    &= \swap{\bm{0}_{1}}{\bm{1}_{1}} \;
    \swap{\bm{0}_{1}}{\bm{10}_{1}} \;
    \swap{\bm{10}_{1}}{\bm{11}_{1}} \;
    \swap{\bm{01}_{1}}{\bm{1}_{1}}
    = X_{0,1} \, \pi_{0}.
  \end{align*}
  The case when $d \geq 2$ now follows using
  Lemma~\ref{lem:baker}(i)--(iii).

  We establish Relation~\eqref{eq:cut7} first in the case when $d' =
  1$ and $m = 0$.  If $d \geq 2$, then:
  \begin{align*}
    X_{0,1} \, X_{1,1} \, X_{0,d}
    &= X_{0,1} \; X_{1,1} \; X_{0,1} \; B_{d}(\bm{1}_{1}) \\
    &= X_{0,1} \; X_{1,1} \; X_{0,1} \;
    \swap{\bm{101}_{1}}{\bm{110}_{1}} \; B_{d}(\bm{10}_{1}) \;
    B_{d}(\bm{11}_{1}) \\
    &= \swap{\bm{001}_{1}}{\bm{01}_{1}} \; X_{0,1} \; X_{1,1} \;
    X_{0,1} \; B_{d}(\bm{10}_{1}) \; B_{d}(\bm{11}_{1})
    &&\text{by repeated use of~\eqref{rel:conj}} \\
    &= \pi_{1} \; X_{0,1} \; X_{1,1} \; B_{d}(\bm{01}_{1}) \;
    B_{d}(\bm{1}_{1}) \; X_{0,1}
    &&\hspace{-40pt} \text{by Lemma~\ref{lem:baker}(i)--(iii)
      and~\eqref{rel:conj}} \\
    &= \pi_{1} \; X_{0,1} \; B_{d}(\bm{1}_{1}) \; X_{1,1} \;
    B_{d}(\bm{01}_{1}) \; X_{0,1}
    &&\hspace*{-40pt}\text{by Lemma~\ref{lem:baker}(i) to
      move~$B_{d}(\bm{1}_{1})$} \\
    &= \pi_{1} \, X_{0,d} \, X_{1,d} \, X_{0,1}.
  \end{align*}
  We can interchange the roles of $d$~and~$d'$ in~\eqref{eq:cut7} by
  multiplying on the left by~$\pi_{m+1}$.  Hence it remains to
  establish the relation in the cases when both $d,d' \geq 2$.  This
  is achieved as follows:
  \begin{align*}
    X_{0,d'} \, X_{1,d'} \, X_{0,d} &= X_{0,d'} \; X_{1,d'} \; X_{0,1}
    \; B_{d}(\bm{1}_{1}) \\
    &= \pi_{1} \; X_{0,1} \; X_{1,1} \; X_{0,d'} \; B_{d}(\bm{1}_{1})
    &&\text{by the case $d = 1$} \\
    &= \pi_{1} \; X_{0,1} \; X_{1,1} \; X_{0,1} \;
    B_{d'}(\bm{1}_{1}) \; B_{d}(\bm{1}_{1}.\bm{0}_{d'}) \;
    B_{d}(\bm{1}_{1}.\bm{1}_{d'})
    &&\text{by Lemma~\ref{lem:baker}(vi)} \\
    &= \pi_{1} \; X_{0,1} \; X_{1,1} \; X_{0,1} \; B_{d}(\bm{10}_{1})
    \; B_{d}(\bm{11}_{1}) \; B_{d'}(\bm{1}_{1})
    &&\text{by Equation~\eqref{eq:baker-conj}} \\
    &= \pi_{1} \; X_{0,1} \; B_{d}(\bm{1}_{1}) \; X_{1,1} \;
    B_{d}(\bm{01}_{1}) \; X_{0,d'}
    &&\text{by Lemma~\ref{lem:baker} and~\eqref{rel:conj}} \\
    &= \pi_{1} \, X_{0,d} \, X_{1,d} \, X_{0,d'}.
  \end{align*}

  Relation~\eqref{eq:bakers14} is established by using formulae about
  conjugation by~$\swap{\bm{0}_{1}}{\bm{1}_{1}}$:
  \begin{align*}
    X_{0,d} \, \bar{\pi}_{0}
    &= \swap{\bm{0}_{1}}{\bm{1}_{1}} \;
    \swap{\bm{0}_{1}}{\bm{10}_{1}} \;
    \swap{\bm{10}_{1}}{\bm{11}_{1}} \; B_{d}(\bm{1}_{1}) \;
    \swap{\bm{0}_{1}}{\bm{1}_{1}} \\
    &= \swap{\bm{00}_{1}}{\bm{1}_{1}} \;
    \swap{\bm{00}_{1}}{\bm{01}_{1}} \; B_{d}(\bm{0}_{1})
    = \swap{\bm{00}_{1}}{\bm{01}_{1}} \;
    \swap{\bm{01}_{1}}{\bm{1}_{1}} \; B_{d}(\bm{0}_{1})
    = \bar{\pi}_{1} \, \pi_{0} \, C_{1,d}.
  \end{align*}
  For~\eqref{eq:bakers16}, we collect the transpositions
  comprising~$X_{0,1}$ to the right:
  \begin{align*}
    X_{0,1} \, C_{0,d}
    &= X_{0,1} \; \swap{\bm{01}_{1}}{\bm{10}_{1}} \;
    B_{d}(\bm{0}_{1}) \; B_{d}(\bm{1}_{1}) \\
    &= \swap{\bm{001}_{1}}{\bm{01}_{1}} \; B_{d}(\bm{00}_{1}) \;
    X_{0,1} \; B_{d}(\bm{1}_{1})
    = \pi_{1} \, C_{2,d} \, X_{0,d}
  \end{align*}
  Finally consider relation~\eqref{eq:bakers18} in the case when $m =
  0$.  We calculate
  \begin{align*}
    C_{2,d'} \, X_{0,d'} \, C_{0,d}
    &= B_{d'}(\bm{00}_{1}) \; X_{0,1} \; B_{d'}(\bm{1}_{1}) \; C_{0,d}
    \\
    &= X_{0,1} \; B_{d'}(\bm{0}_{1}) \; B_{d'}(\bm{1}_{1}) \; C_{0,d}
    &&\text{by~\eqref{rel:conj} and Lemma~\ref{lem:baker}(i),~(iii)}
    \\
    &= X_{0,1} \; C_{0,d} \; C_{0,d'} &&\text{by
      Lemma~\ref{lem:fullbaker}(i), (ii)} \\
    &= X_{0,1} \; \swap{\bm{01}_{1}}{\bm{10}_{1}} \; B_{d}(\bm{0}_{1})
    \; B_{d}(\bm{1}_{1}) \; C_{0,d'} \\
    &= \swap{\bm{001}_{1}}{\bm{01}_{1}} \; B_{d}(\bm{00}_{1}) \;
    X_{0,1} \; B_{d}(\bm{1}_{1}) \; C_{0,d'}
    &&\text{by~\eqref{rel:conj} and Lemma~\ref{lem:baker}(i),~(iii)}
    \\
    &= \pi_{1} \, C_{2,d} \, X_{0,d} \, C_{0,d'},
  \end{align*}
  as required.  This completes the proof of the proposition and the
  work of this section.
\end{proof}

\section{Finite presentations for~\boldmath$nV$}
\label{sec:finitepres}

In this section, we prove Theorem~\ref{thm:finitepres}.  In the course
of our argument, we shall refer to two particular subsets of the
collection~$\Omega$ of addresses.  Write $\Delta = (X^{2})^{n}$ for
the set of addresses $\bm{\alpha} =
(\alpha_{1},\alpha_{2},\dots,\alpha_{n})$ satisfying
$\length{\alpha_{i}} = 2$ for $i = 1$,~$2$, \dots,~$n$.  Also
define~$\Omega^{\ast}$ to be the set of addresses $\bm{\alpha} =
(\alpha_{1},\alpha_{2},\dots,\alpha_{n})$ satisfying
$\length{\alpha_{i}} \geq 2$ for $i = 1$,~$2$, \dots,~$n$.

We shall define a presentation for a group~$G$ on three generators
$a$,~$b$ and~$c$ and $2n^{2} + 3n + 11$ relations.  The starting point
is a presentation for the symmetric group of degree~$4^{n}$ on two
generators.  According to~\cite[Theorem~A]{GKKL}, this can be achieved
using merely eight relations (independent of the value of~$n$).  A
recent article by Huxford~\cite{Huxford} presents corrections to the
arguments and the relations given in~\cite{GKKL}.  Note, however, that
the number of relations is unaffected and the statement
of~\cite[Theorem~A]{GKKL} remains valid.

Our first family (\ref{rel:R1-Sym} below) of relations is sufficient
to ensure that the generators $a$~and~$b$ satisfy all relations that
hold in the symmetric group of degree~$4^{n}$.  Moreover, all the
relations that we assume are satisfied if one maps $a$,~$b$ and~$c$ to
the corresponding elements of~$nV$ (where for~$c$ we interpret
Relation~\ref{rel:R7-c-def} in~$nV$).  In particular, the resulting
homomorphism $G \to nV$ induces a homomorphism from $H = \langle a,b
\rangle$ onto the above symmetric group.  Consequently, $H \cong
\Sym(\Delta)$ and we may interpret the elements in $H$ as defining
permutations of~$\Delta$.  We therefore use the
symbol~$\swap{\bm{\alpha}}{\bm{\beta}}$, where $\bm{\alpha},
\bm{\beta} \in \Delta$ with $\bm{\alpha} \perp \bm{\beta}$, to denote
certain elements of the subgroup~$H$ and more generally refer to
\emph{permutations} of~$\Delta$ by which we mean the corresponding
elements of this subgroup.  This also means that we can speak of the
support of an element $g \in H$ and use the notation $\bm{\gamma}
\bullet g$ to denote the effect of applying~$g$ to some
address~$\bm{\gamma} \in \Delta$.  (In some sense, this extends the
notation given in Section~\ref{sec:intro}.)

The third generator~$c$ will be used to construct further
transpositions~$\swap{\bm{\alpha}}{\bm{\beta}}$ for other addresses
$\bm{\alpha}, \bm{\beta} \in \Omega^{\ast}$ with $\bm{\alpha} \perp
\bm{\beta}$.  The details of this construction will be described later
in this section, together with appropriate verifications that the
resulting elements of~$G$ are well-defined and satisfy the
relations~\eqref{rel:order}--\eqref{rel:split} listed in
Theorem~\ref{thm:infpres}.  For the remaining discussion, prior to
explaining the construction, we shall assume the existence of the
various transpositions~$\swap{\bm{\alpha}}{\bm{\beta}}$, each of which
will be expressed as some product in~$G$ involving the generators
$a$,~$b$ and~$c$.

To specify the relations that we assume, we shall fix certain elements
of~$H$.  Let $\bm{\delta}^{(0)}$,~$\bm{\delta}^{(1)}$, \dots,
$\bm{\delta}^{(4^{n}-1)}$ be an enumeration of the addresses
in~$\Delta$ and define the following elements of~$H$:
\begin{equation}
  \begin{aligned}
    p &= ( \bm{\delta}^{(n+1)} \; \bm{\delta}^{(n+2)} \, \dots \,
    \bm{\delta}^{(4^{n}-1)} ) \\
    q_{d} &= ( \bm{\delta}^{(0)} \; \bm{\delta}^{(1)} \, \dots \,
    \bm{\delta}^{(d-1)} \; \bm{\delta}^{(d+1)} \, \dots \,
    \bm{\delta}^{(n)} \; \bm{\delta}^{(n+2)} \, \dots \,
    \bm{\delta}^{(4^{n}-1)} )
  \end{aligned}
  \label{eq:perms1}
\end{equation}
The relations that we assume are the following:
\begin{enumerate}
  \renewcommand{\theenumi}{R\arabic{enumi}}
  \renewcommand{\labelenumi}{\theenumi.}
\item \label{rel:R1-Sym}
  eight relations involving $a$~and~$b$ that define the symmetric
  group of degree~$4^{n}$ (from~\cite[Theorem~A]{GKKL});
\item \label{rel:R2-c-disj}
  $[c,p] = \bigl[\, c, \,
  \swap{\bm{\delta}^{(n+2)}}{\bm{\delta}^{(n+3)}} \, \bigr] = 1$;
\item \label{rel:R3-q-disj}
  $\bigl[ \, q_{d}, \,
  \swap{\bm{\delta}^{(d)}.\bm{x}_{d}}{\bm{\delta}^{(n+1)}} \, \bigr] =
  1$ \ for $d = 1$,~$2$, \dots,~$n$ and $x \in \{0,1\}$;
\item \label{rel:R4-c-disj2}
  $\bigl[ \, c, \,
  \swap{\bm{\delta}^{(n+2)}.\bm{x}_{d}}{\bm{\delta}^{(n+3)}} \, \bigr]
  = 1$ \ for $d = 1$,~$2$, \dots,~$n$ and $x \in \{0,1\}$;
\item \label{rel:R5-welldef}
  $ \swap{\bm{\delta}^{(0)}.\bm{x}_{d}}{\bm{\delta}^{(1)}}^{%
  \swap{\bm{\delta}^{(0)}}{\bm{\delta}^{(2)}.\bm{y}_{d'}}} =
  \swap{\bm{\delta}^{(0)}.\bm{y}_{d'}}{\bm{\delta}^{(1)}}^{%
    \swap{\bm{\delta}^{(0)}}{\bm{\delta}^{(2)}.\bm{x}_{d}}} $
  
  for all choices of distinct indices $d,d' \in \{1,2,\dots,n\}$ and
  all pairs $x,y \in \{0,1\}$;
\item \label{rel:R6-split}
  $ \swap{\bm{\delta}^{(0)}}{\bm{\delta}^{(1)}} =
  \swap{\bm{\delta}^{(0)}.\bm{0}_{d}}{\bm{\delta}^{(1)}.\bm{0}_{d}}
  \;
  \swap{\bm{\delta}^{(0)}.\bm{1}_{d}}{\bm{\delta}^{(1)}.\bm{1}_{d}} $
  \ for $d = 1$,~$2$, \dots,~$n$; and
\item \label{rel:R7-c-def}
  $c = \swap{\bm{\delta}^{(0)}}{\bm{\delta}^{(1)}.\bm{0}_{1}} \;
  \swap{\bm{\delta}^{(0)}}{\bm{\delta}^{(1)}.\bm{1}_{1}} \;
  \swap{\bm{\delta}^{(0)}}{\bm{\delta}^{(2)}.\bm{0}_{2}} \;
  \swap{\bm{\delta}^{(0)}}{\bm{\delta}^{(2)}.\bm{1}_{2}}$ \\[5pt]
  \mbox{}\hfill$\cdot \ldots \,
  \swap{\bm{\delta}^{(0)}}{\bm{\delta}^{(n)}.\bm{0}_{n}} \;
  \swap{\bm{\delta}^{(0)}}{\bm{\delta}^{(n)}.\bm{1}_{n}}$.\qquad\mbox{}
\end{enumerate}
We shall establish Theorem~\ref{thm:finitepres} by showing that the
group~$G$, generated by $a$,~$b$ and~$c$ subject to the $2n^{2} + 3n +
11$ relations listed in~\ref{rel:R1-Sym}--\ref{rel:R7-c-def}, is
isomorphic to the group~$nV$.

The first part of the proof of Theorem~\ref{thm:finitepres} involves
the definition of all transpositions~$\swap{\bm{\alpha}}{\bm{\beta}}$,
for $\bm{\alpha}, \bm{\beta} \in \Omega^{\ast}$ with $\bm{\alpha}
\perp \bm{\beta}$, as elements of the group~$G$, verifying that these
transpositions are well-defined, and that they satisfy all those
Relations~\eqref{rel:order}--\eqref{rel:split} listed in
Theorem~\ref{thm:infpres} involving only addresses in~$\Omega^{\ast}$.
(In this part of the proof we will not consider any address
$\bm{\alpha} = (\alpha_{1},\alpha_{2},\dots,\alpha_{n})$ with one or
more coordinates satisfying $\length{\alpha_{i}} \leq 1$.)  We shall
also verify explicitly the relation listed in
Equation~\eqref{rel:symm} since we depend upon this to reduce the
number of cases when verifying Relation~\eqref{rel:conj} (see
Lemma~\ref{lem:conj-types} below).  The overall process is an
induction argument based on the number of coordinates of an
address~$\bm{\alpha}$ having some specified length.  To make this
precise, we set $m(\bm{\alpha}) = \max \bigl\{ \length{\alpha_{1}},
\length{\alpha_{2}}, \dots, \length{\alpha_{n}} \bigr\}$ and define
$k(\bm{\alpha})$~to be the number of coordinates~$\alpha_{i}$
satisfying $\length{\alpha_{i}} = m(\bm{\alpha})$.  The \emph{weight}
of~$\bm{\alpha}$ is then the ordered pair $\wt(\bm{\alpha}) =
(m(\bm{\alpha}), k(\bm{\alpha}))$.  This pair then measures the
longest coordinate of~$\bm{\alpha}$ and counts the number of longest
coordinates in the address.  Thus, for example, $\Delta$~consists of
all addresses of weight~$(2,n)$.

We order weights lexicographically, so $\wt(\bm{\alpha}) <
\wt(\bm{\beta})$ means either $m(\bm{\alpha}) <
m(\bm{\beta})$, or $m(\bm{\alpha}) = m(\bm{\beta})$ and
$k(\bm{\alpha}) < k(\bm{\beta})$; that is, either the coordinates
of~$\bm{\alpha}$ are all shorter than the longest of~$\bm{\beta}$ or
that the greatest length is the same but $\bm{\alpha}$~has fewer of
these longest coordinates than~$\bm{\beta}$.

In each step of the induction, we assume that we already have defined
transpositions whose entries have weight less than~$(m,k)$ and
verified all Relations~\eqref{rel:order}--\eqref{rel:split}, and
also~\eqref{rel:symm}, involving such transpositions.  Our first stage
is then to define transpositions~$\swap{\bm{\alpha}}{\bm{\beta}}$
where $\wt(\bm{\alpha}) = (m,k)$ and $\bm{\beta} \in \Delta$, or
\emph{vice versa}.  We then verify that our definitions make sense and
that all the required relations involving the newly defined
transpositions are satisfied.  At the second stage, we perform the
same definitions and verifications for the remaining
transpositions~$\swap{\bm{\alpha}}{\bm{\beta}}$ with
$\wt(\bm{\alpha}), \wt(\bm{\beta}) \leq (m,k)$.  The new
transpositions will always be conjugates of transpositions
from~$\Sym(\Delta)$ and consequently Relation~\eqref{rel:order} will
always hold and we do not verify it explicitly.

A significant part of our argument is verifying the conjugacy
relations~\eqref{rel:conj}.  The following lemma describes which
instances of this relation are needed at each step of the induction.

\begin{lemma}
  \label{lem:conj-types}
  Assume that all transpositions defined at some stage of the
  induction satisfy $\swap{\bm{\alpha}}{\bm{\beta}} =
  \swap{\bm{\beta}}{\bm{\alpha}}$ and
  $\swap{\bm{\alpha}}{\bm{\beta}}^{2} = 1$.  To verify all required
  conjugacy relations~\eqref{rel:conj} it is sufficient to establish
  them in the following cases:
  \begin{enumerate}
    \renewcommand{\theenumi}{(\Alph{enumi})}
  \item Every pair from $\{ \bm{\alpha}, \bm{\beta}, \bm{\gamma},
    \bm{\delta} \}$ is incomparable.  The required relation then has
    the form
    \[
    \bigl[ \, \swap{\bm{\alpha}}{\bm{\beta}} \, , \,
      \swap{\bm{\gamma}}{\bm{\delta}} \, \bigr] = 1.
    \]
    
  \item $\bm{\gamma} \preceq \bm{\alpha}$ and every pair from $\{
    \bm{\beta}, \bm{\gamma}, \bm{\delta} \}$ is incomparable.  Then
    $\bm{\alpha} = \bm{\gamma}\bm{\eta}$ for some $\bm{\eta} \in
    \Omega$ and the required relation is
    \[
    \swap{\bm{\gamma\eta}}{\bm{\beta}}^{\swap{\bm{\gamma}}{\bm{\delta}}}
    = \swap{\bm{\delta\eta}}{\bm{\beta}}.
    \]

  \item $\bm{\gamma} \prec \bm{\alpha}$, \ $\bm{\delta} \preceq
    \bm{\beta}$ and $\bm{\gamma} \perp \bm{\delta}$.  Then
    $\bm{\alpha} = \bm{\gamma\eta}$ and $\bm{\beta} =
    \bm{\delta\theta}$ for some $\bm{\eta},\bm{\theta} \in \Omega$ and
    the required relation is
    \[
    \swap{\bm{\gamma\eta}}{\bm{\delta\theta}}^{%
      \swap{\bm{\gamma}}{\bm{\delta}}} =
    \swap{\bm{\gamma\theta}}{\bm{\delta\eta}}.
    \]

  \item $\bm{\gamma} \prec \bm{\alpha}, \bm{\beta}$ and $\bm{\gamma}
    \perp \bm{\delta}$.  Then $\bm{\alpha} = \bm{\gamma\eta}$ and
    $\bm{\beta} = \bm{\gamma\theta}$ for (necessarily non-empty)
    $\bm{\eta}, \bm{\theta} \in \Omega$ and the required relation is
    \[
    \swap{\bm{\gamma\eta}}{\bm{\gamma\theta}}^{%
      \swap{\bm{\gamma}}{\bm{\delta}}} =
    \swap{\bm{\delta\eta}}{\bm{\delta\theta}}.
    \]
  \end{enumerate}
\end{lemma}

\begin{proof}
  In order that $\bm{\alpha} \bullet \swap{\bm{\gamma}}{\bm{\delta}}$
  be defined, we require (i)~$\bm{\gamma} \perp \bm{\alpha}$ or
  $\bm{\gamma} \preceq \bm{\alpha}$, and (ii)~$\bm{\delta} \perp
  \bm{\alpha}$ or $\bm{\delta} \preceq \bm{\alpha}$.  A similar pair
  of conditions apply when $\bm{\beta} \bullet
  \swap{\bm{\gamma}}{\bm{\delta}}$ is defined.  We analyze the four
  resulting conditions.  Furthermore, we note, for example,
  that if $\bm{\gamma} \preceq \bm{\alpha}$ or $\bm{\delta} \preceq
  \bm{\alpha}$, then by exploiting the symmetry
  $\swap{\bm{\gamma}}{\bm{\delta}} = \swap{\bm{\delta}}{\bm{\gamma}}$
  we may assume that in fact $\bm{\gamma} \preceq \bm{\alpha}$.  This
  reduces the four conditions to the configurations described in the
  statement of the lemma.
\end{proof}

\paragraph{Base Step:} As the base step of the induction, we rely upon
the assumed Relations~\ref{rel:R1-Sym} to ensure that the
transpositions in $H = \langle a,b \rangle$ exist and satisfy all the
relations  in the symmetric group of degree~$4^{n}$; that is, we may
use all transpositions~$\swap{\bm{\alpha}}{\bm{\beta}}$, where
$\bm{\alpha}, \bm{\beta} \in \Omega^{\ast}$ satisfy $\bm{\alpha} \perp
\bm{\beta}$ and $\wt(\bm{\alpha}) = \wt(\bm{\beta}) = (2,n)$, and all
relations~\eqref{rel:order}, \eqref{rel:conj} and~\eqref{rel:symm}
between them.  (There are no relations of the form~\eqref{rel:split}
involving only transpositions from~$H$.)

\paragraph{Induction, Stage~1:} Let us assume that, for some fixed
weight $(m,k) \geq (3,1)$, we have already defined all
transpositions~$\swap{\bm{\alpha}}{\bm{\beta}}$ where $\bm{\alpha},
\bm{\beta} \in \Omega^{\ast}$ satisfy $\bm{\alpha} \perp \bm{\beta}$
and $\wt(\bm{\alpha}), \wt(\bm{\beta}) < (m,k)$ and verified all
relations~\eqref{rel:order}--\eqref{rel:symm} involving such
transpositions.  We now define those tranpositions with one entry of
weight~$(m,k)$ and one entry from~$\Delta$.  The definitions and
argument actually depend upon whether $(m,k) = (3,1)$ or $(m,k) >
(3,1)$.  Consequently, we split into these cases.

\spc

Suppose then that $(m,k) = (3,1)$.  First set
\begin{equation}
  \swap{\bm{\delta}^{(d)}.\bm{x}_{d}}{\bm{\delta}^{(n+1)}} \defeq
  \swap{\bm{\delta}^{(0)}}{\bm{\delta}^{(n+1)}}^{c^{2d+x-1}}
  \label{eq:31-basicdef}
\end{equation}
where $d = 1$,~$2$, \dots,~$n$ and $x \in \{0,1\}$.  (Since
$d$~and~$x$ are non-negative integers, the sum $2d+x-1$ is defined.)
We now define transpositions~$\swap{\bm{\alpha}}{\bm{\beta}}$ where
$\bm{\alpha}, \bm{\beta} \in \Omega^{\ast}$ satisfy $\bm{\alpha} \perp
\bm{\beta}$, \ $\wt(\bm{\alpha}) = (3,1)$ and $\bm{\beta} \in
\Delta$.  Then $\bm{\alpha} =
(\alpha_{1},\alpha_{2},\dots,\alpha_{n})$ where $\length{\alpha_{d}} =
3$ for some index~$d$ and $\length{\alpha_{i}} = 2$ for all other
indices~$i$.  Write $\bm{\alpha} = \hat{\bm{\alpha}}.\bm{x}_{d}$,
where $\hat{\bm{\alpha}} \in \Delta$, and choose a permutation~$\sigma
\in H$ that moves $\bm{\delta}^{(d)}$ to~$\hat{\bm{\alpha}}$ and
$\bm{\delta}^{(n+1)}$ to~$\bm{\beta}$.  We then define
\begin{equation}
  \swap{\bm{\alpha}}{\bm{\beta}} \defeq
  \swap{\bm{\delta}^{(d)}.\bm{x}_{d}}{\bm{\delta}^{(n+1)}}^{\sigma}
  \label{eq:31-swap}
\end{equation}
in terms of the transposition defined in~\eqref{eq:31-basicdef}.  To
achieve the symmetry required by the relation~\eqref{rel:symm}, we also define
$\swap{\bm{\beta}}{\bm{\alpha}} \defeq \swap{\bm{\alpha}}{\bm{\beta}}$
for such $\bm{\alpha}$~and~$\bm{\beta}$.  To verify that these
definitions are independent of the choice of permutation $\sigma \in
H$, we use the following lemma.

\begin{lemma}
  \label{lem:31-def}
  \begin{enumerate}
  \item If $\sigma \in H$~is a permutation of~$\Delta$ with support
    disjoint from $\bm{\delta}^{(0)}$,~$\bm{\delta}^{(1)}$,
    \dots,~$\bm{\delta}^{(n)}$, then $[c,\sigma] = 1$.
  \item Let $d = 1$,~$2$, \dots,~$n$ and $x \in \{0,1\}$.  Then
    $\swap{\bm{\delta}^{(d)}.\bm{x}_{d}}{\bm{\delta}^{(n+1)}}$
    commutes with every permutation in~$H$ that has support disjoint
    from $\bm{\delta}^{(d)}$~and~$\bm{\delta}^{(n+1)}$.
  \end{enumerate}
\end{lemma}

\begin{proof}
  (i)~Note that our permutation~$p$, defined in~\eqref{eq:perms1},
  together with~$\swap{\bm{\delta}^{(n+2)}}{\bm{\delta}^{(n+3)}}$
  generate all the permutations of $\Delta \setminus \{
  \bm{\delta}^{(0)}, \bm{\delta}^{(1)}, \dots, \bm{\delta}^{(n)} \}$.
  Hence all permutations of the latter set commute with~$c$ by our
  assumed relations~\ref{rel:R2-c-disj}.

  (ii)~The disjoint transpositions
  $\swap{\bm{\delta}^{(0)}}{\bm{\delta}^{(n+1)}}$ and
  $\swap{\bm{\delta}^{(n+2)}}{\bm{\delta}^{(n+3)}}$ commute.
  Consequently
  $\swap{\bm{\delta}^{(d)}.\bm{x}_{d}}{\bm{\delta}^{(n+1)}} =
  \swap{\bm{\delta}^{(0)}}{\bm{\delta}^{(n+1)}}^{c^{2d+x-1}}$ commutes
  with $\swap{\bm{\delta}^{(n+2)}}{\bm{\delta}^{(n+3)}}$
  by~\ref{rel:R2-c-disj}.  The definition~\eqref{eq:perms1}  of the
  cycle~$q_{d}$ ensures that it together with
  $\swap{\bm{\delta}^{(n+2)}}{\bm{\delta}^{(n+3)}}$ generate all
  permutations in~$H$ with support disjoint from
  $\bm{\delta}^{(d)}$~and~$\bm{\delta}^{(n+1)}$.  Hence
  $\swap{\bm{\delta}^{(d)}.\bm{x}_{d}}{\bm{\delta}^{(n+1)}}$~commutes
  with all such permutations with use of
  Relation~\ref{rel:R3-q-disj}.
\end{proof}

We now show that the definition of~$\swap{\bm{\alpha}}{\bm{\beta}}$ in
Equation~\eqref{eq:31-swap} does not depend on the choice of~$\sigma$.
For if $\sigma_{1}$~and~$\sigma_{2}$ are two permutations of~$\Delta$
that both move $\bm{\delta}^{(d)}$ to~$\hat{\bm{\alpha}}$ and
$\bm{\delta}^{(n+1)}$ to~$\bm{\beta}$, then
$\sigma_{1}\sigma_{2}^{-1}$~fixes both
$\bm{\delta}^{(d)}$~and~$\bm{\delta}^{(n+1)}$, so commutes with
$\swap{\bm{\delta}^{(d)}.\bm{x}_{d}}{\bm{\delta}^{(n+1)}}$ by
part~(ii) of the lemma.  Therefore
$\swap{\bm{\delta}^{(d)}.\bm{x}_{d}}{\bm{\delta}^{(n+1)}}^{\sigma_{1}}
=
\swap{\bm{\delta}^{(d)}.\bm{x}_{d}}{\bm{\delta}^{(n+1)}}^{\sigma_{2}}$,
establishing that $\swap{\bm{\alpha}}{\bm{\beta}}$~is well-defined.

There are no split relations~\eqref{rel:split} to verify at this
stage, since we have not yet introduced any
transpositions~$\swap{\bm{\alpha}}{\bm{\beta}}$ where both
$\bm{\alpha}$~and~$\bm{\beta}$ have coordinates of length~$3$.  Hence
we only need to check conjugacy relations~\eqref{rel:conj} involving
the transpositions that we have introduced.

\begin{lemma}
  \label{lem:31^perm}
  \begin{enumerate}
  \item Let $\bm{\beta},\bm{\gamma} \in \Delta$, \ $x \in \{0,1\}$ and
    $d$~be an index with $1 \leq d \leq n$.  If $\sigma \in
    \Sym(\Delta)$, then
    \[
    \swap{\bm{\beta}.\bm{x}_{d}}{\bm{\gamma}}^{\sigma} =
    \swap{(\bm{\beta}\!\bullet\!\sigma).\bm{x}_{d}}{\bm{\gamma}\!\bullet\!\sigma},
    \]
    where, as above,
    $\bm{\beta}\bullet\sigma$~and~$\bm{\gamma}\bullet\sigma$ denote
    the images of $\bm{\beta}$~and~$\bm{\gamma}$ under the action
    of~$\sigma \in H$.
  \item If $\bm{\beta} \in \Delta \setminus \{ \bm{\delta}^{(0)},
    \bm{\delta}^{(1)}, \dots, \bm{\delta}^{(n)} \}$, \ $1 \leq d \leq
    n$ and $x \in \{0,1\}$, then
    \[
    \swap{\bm{\delta}^{(0)}}{\bm{\beta}}^{c^{2d+x-1}} =
    \swap{\bm{\delta}^{(d)}.\bm{x}_{d}}{\bm{\beta}}.
    \]
  \end{enumerate}
\end{lemma}

\begin{proof}
  (i)~This follows immediately from the definition of transpositions
  of the form~$\swap{\bm{\beta}.\bm{x}_{d}}{\bm{\gamma}}$ and their
  well-definedness.

  (ii)~We must establish the result when $\bm{\beta} \neq
  \bm{\delta}^{(n+1)}$.  Take $\sigma =
  \swap{\bm{\delta}^{(n+1)}}{\bm{\beta}}$.  This commutes with~$c$ by
  Lemma~\ref{lem:31-def}(i).  Hence
  \[
  \swap{\bm{\delta}^{(0)}}{\bm{\beta}}^{c^{2d+x-1}} =
  \swap{\bm{\delta}^{(0)}}{\bm{\delta}^{(n+1)}}^{\sigma \, c^{2d+x-1}}
  = \swap{\bm{\delta}^{(0)}}{\bm{\delta}^{(n+1)}}^{c^{2d+x-1} \,
    \sigma} =
  \swap{\bm{\delta}^{(d)}.\bm{x}_{d}}{\bm{\delta}^{(n+1)}}^{\sigma},
  \]
  which produces the required result using part~(i).
\end{proof}

Part~(i) of Lemma~\ref{lem:31^perm} establishes any instance of the
conjugacy relation~\eqref{rel:conj} when
$\swap{\bm{\gamma}}{\bm{\delta}} \in H$.  We consider now the
remaining instances of~\eqref{rel:conj} involving transpositions
defined at this stage of the induction (via~\eqref{eq:31-swap} above)
and we split into the cases listed in Lemma~\ref{lem:conj-types}.

\textbf{(A):}~Consider four incomparable addresses
$\bm{\alpha}$,~$\bm{\beta}$, $\bm{\gamma}$,~$\bm{\delta}$ where we may
assume $\wt(\bm{\alpha}) = \wt(\bm{\gamma}) = (3,1)$ and $\bm{\beta},
\bm{\delta} \in \Delta$.  Write $\bm{\alpha} =
\hat{\bm{\alpha}}.\bm{x}_{d}$ and $\bm{\gamma} =
\hat{\bm{\gamma}}.\bm{y}_{d'}$ for $x,y \in \{0,1\}$, indices
$d$~and~$d'$, and some $\hat{\bm{\alpha}}, \hat{\bm{\gamma}} \in
\Delta$.  Suppose $\hat{\bm{\alpha}} \neq \hat{\bm{\gamma}}$, so that
$\hat{\bm{\alpha}}$,~$\bm{\beta}$, $\hat{\bm{\gamma}}$
and~$\bm{\delta}$ are distinct addresses in~$\Delta$.  From
Lemma~\ref{lem:31^perm}(i), we know
$\swap{\bm{\delta}^{(n+2)}.\bm{x}_{d}}{\bm{\delta}^{(n+3)}}$ commutes
with $\swap{\bm{\delta}^{(0)}}{\bm{\delta}^{(n+1)}}$ and it commutes
with~$c$ by~\ref{rel:R4-c-disj2}.  Hence, conjugating
by~$c^{2d'+y-1}$, we conclude that
\[
\bigl[ \swap{\bm{\delta}^{(n+2)}.\bm{x}_{d}}{\bm{\delta}^{(n+3)}} ,
  \swap{\bm{\delta}^{(d')}.\bm{y}_{d'}}{\bm{\delta}^{(n+1)}} \bigr] = 1.
\]
Finally conjugate by a permutation~$\sigma$ in~$H$ that maps
$\bm{\delta}^{(n+2)}$ to~$\hat{\bm{\alpha}}$, \ $\bm{\delta}^{(n+3)}$
to~$\bm{\beta}$, \ $\bm{\delta}^{(d')}$ to~$\hat{\bm{\gamma}}$ and
$\bm{\delta}^{(n+1)}$ to~$\bm{\delta}$ to establish $\bigl[
  \swap{\bm{\alpha}}{\bm{\beta}} , \swap{\bm{\gamma}}{\bm{\delta}}
  \bigr] = 1$, as required.

On the other hand, if $\hat{\bm{\alpha}} = \hat{\bm{\gamma}}$, then in
order that $\bm{\alpha} \perp \bm{\gamma}$, we know that $d = d'$ and,
after suitable relabelling, we can assume $x = 0$ and~$y = 1$.  We
make use of Lemma~\ref{lem:31^perm}(ii) to observe:
\begin{align*}
  \swap{\bm{\delta}^{(d)}.\bm{0}_{d}}{\bm{\delta}^{(n+1)}} &=
  \swap{\bm{\delta}^{(0)}}{\bm{\delta}^{(n+1)}}^{c^{2d-1}}
  \\
  \swap{\bm{\delta}^{(d)}.\bm{1}_{d}}{\bm{\delta}^{(n+2)}} &=
  \swap{\bm{\delta}^{(0)}}{\bm{\delta}^{(n+2)}}^{c^{2d}} =
  \swap{\bm{\delta}^{(1)}.\bm{0}_{1}}{\bm{\delta}^{(n+2)}}^{c^{2d-1}}.
\end{align*}
By Lemma~\ref{lem:31^perm}(i),
$\swap{\bm{\delta}^{(0)}}{\bm{\delta}^{(n+1)}}$ commutes with
$\swap{\bm{\delta}^{(1)}.\bm{0}_{1}}{\bm{\delta}^{(n+2)}}$.  Hence,
upon conjugating by~$c^{2d-1}$, we deduce
\[
[ \swap{\bm{\delta}^{(d)}.\bm{0}_{d}}{\bm{\delta}^{(n+1)}} ,
  \swap{\bm{\delta}^{(d)}.\bm{1}_{d}}{\bm{\delta}^{(n+2)}} ] = 1.
\]
We establish the required relation $[ \swap{\bm{\alpha}}{\bm{\beta}} ,
  \swap{\bm{\gamma}}{\bm{\delta}} ] = 1$ by finally conjugating by a
permutation~$\sigma$ in~$H$ that moves $\bm{\delta}^{(d)}$
to~$\hat{\bm{\alpha}}$, \ $\bm{\delta}^{(n+1)}$ to~$\bm{\beta}$
and~$\bm{\delta}^{(n+2)}$ to~$\bm{\delta}$.

\textbf{(B):}~In view of Lemma~\ref{lem:31^perm}(i) and the symmetry
between $\bm{\gamma}$~and~$\bm{\delta}$ in the conjugacy relation,
we may assume in this case that $\wt(\bm{\gamma}) = (3,1)$ and
$\bm{\alpha} = \bm{\gamma}$.  The remaining addresses
$\bm{\beta}$~and~$\bm{\delta}$ must be from~$\Delta$.  Write
$\bm{\gamma} = \hat{\bm{\gamma}}.\bm{x}_{d}$ for some $x \in \{0,1\}$
and some index~$d$.  Conjugate the equation $
\swap{\bm{\delta}^{(0)}}{\bm{\delta}^{(n+1)}}^{%
  \swap{\bm{\delta}^{(0)}}{\bm{\delta}^{(n+2)}}} =
\swap{\bm{\delta}^{(n+1)}}{\bm{\delta}^{(n+2)}}$ by~$c^{2d+x-1}$ to
produce
\[
\swap{\bm{\delta}^{(d)}.\bm{x}_{d}}{\bm{\delta}^{(n+1)}}^{%
  \swap{\bm{\delta}^{(d)}.\bm{x}_{d}}{\bm{\delta}^{(n+2)}}} =
\swap{\bm{\delta}^{(n+1)}}{\bm{\delta}^{(n+2)}}
\]
by use of Lemma~\ref{lem:31-def}(i) and~\ref{lem:31^perm}(ii).
Finally conjugate by a permutation~$\sigma$ of~$\Delta$ moving
$\bm{\delta}^{(d)}$ to~$\hat{\bm{\gamma}}$, \ $\bm{\delta}^{(n+1)}$
to~$\bm{\beta}$ and $\bm{\delta}^{(n+2)}$ to~$\bm{\delta}$, using
Lemma~\ref{lem:31^perm}(i), to establish the required relation.

\textbf{(C)/(D):}~Conjugacy relations of the form~(C) occur only at
this stage when $\bm{\gamma}, \bm{\delta} \in \Delta$, which all then
hold by Lemma~\ref{lem:31^perm}.  There are no Type~(D) relations to
verify.  This completes the verifications required for Stage~1 of the
induction when $(m,k) = (3,1)$.

\spc

Now assume $(m,k) > (3,1)$.  Consider a pair of incomparable addresses
$\bm{\alpha}, \bm{\beta} \in \Omega^{\ast}$ with $\wt(\bm{\alpha}) =
(m,k)$ and $\bm{\beta} \in \Delta$.  To define the
transposition~$\swap{\bm{\alpha}}{\bm{\beta}}$, first choose $d$~to be
the index of one of the coordinates of~$\bm{\alpha}$ that has
length~$m$ and write $\bm{\alpha} = \hat{\bm{\alpha}}.\bm{x}_{d}$
where $x \in \{0,1\}$.  Then $\hat{\bm{\alpha}}$~is an address with
$\wt(\hat{\bm{\alpha}}) < (m,k)$ and, since every coordinate
of~$\bm{\beta}$ has length~$2$, it follows $\hat{\bm{\alpha}} \perp
\bm{\beta}$ also.  Choose $\bm{\zeta} \in \Delta$ incomparable with both
$\hat{\bm{\alpha}}$~and~$\bm{\beta}$.  Then
$\wt(\bm{\zeta}.\bm{x}_{d}) = (3,1) < (m,k)$ and, by induction, both
transpositions $\swap{\bm{\zeta}.\bm{x}_{d}}{\bm{\beta}}$
and~$\swap{\bm{\zeta}}{\hat{\bm{\alpha}}}$ have been constructed.  We
then define
\begin{equation}
  \swap{\bm{\alpha}}{\bm{\beta}} =
  \swap{\hat{\bm{\alpha}}.\bm{x}_{d}}{\bm{\beta}} \defeq
  \swap{\bm{\zeta}.\bm{x}_{d}}{\bm{\beta}}^{%
    \swap{\bm{\zeta}}{\hat{\bm{\alpha}}}}.
  \label{eq:inductionI-def}
\end{equation}
In order to achieve Relation~\eqref{rel:symm}, we also set
$\swap{\bm{\beta}}{\bm{\alpha}} \defeq
\swap{\bm{\alpha}}{\bm{\beta}}$.  We must verify that the above
definition is independent of the choice of the address~$\bm{\zeta}$
and of the index~$d$.

With the above assumptions, we make the following observations:

\begin{lemma}
  \label{lem:induct-welldef}
  \begin{enumerate}
  \item Suppose that $\bm{\zeta}$~and~$\bm{\eta}$ are distinct
    addresses in~$\Delta$ that are incomparable with both
    $\hat{\bm{\alpha}}$~and~$\bm{\beta}$.  Then
    \[
    \swap{\bm{\zeta}.\bm{x}_{d}}{\bm{\beta}}^{%
      \swap{\bm{\zeta}}{\bm{\hat{\alpha}}}} =
    \swap{\bm{\eta}.\bm{x}_{d}}{\bm{\beta}}^{%
      \swap{\bm{\eta}}{\bm{\hat{\alpha}}}}.
    \]
  \item Suppose that $d$~and~$d'$ are both indices of coordinates
    of~$\bm{\alpha}$ of length~$m$.  Write $\bm{\alpha} =
    \bm{\gamma}.\bm{x}_{d}.\bm{y}_{d'}$ for some $x,y \in \{0,1\}$.
    If $\bm{\zeta}$~is an address in~$\Delta$ incomparable with both
    $\bm{\beta}$~and~$\bm{\gamma}$, then
    \[
    \swap{\bm{\zeta}.\bm{x}_{d}}{\bm{\beta}}^{%
      \swap{\bm{\zeta}}{\bm{\gamma}.\bm{y}_{d'}}} =
    \swap{\bm{\zeta}.\bm{y}_{d'}}{\bm{\beta}}^{%
      \swap{\bm{\zeta}}{\bm{\gamma}.\bm{x}_{d}}}.
    \]
  \end{enumerate}
\end{lemma}

\begin{proof}
  (i)~As distinct addresses in~$\Delta$, certainly
  $\bm{\zeta}$~and~$\bm{\eta}$ are incomparable.  All addresses
  appearing in the following calculation have weight $< (m,k)$ and so,
  by induction,
  \[
  \swap{\bm{\zeta}.\bm{x}_{d}}{\bm{\beta}}^{%
    \swap{\bm{\zeta}}{\hat{\bm{\alpha}}} \,
    \swap{\bm{\eta}}{\hat{\bm{\alpha}}}} =
  \swap{\bm{\zeta}.\bm{x}_{d}}{\bm{\beta}}^{%
    \swap{\bm{\eta}}{\bm{\zeta}} \,
    \swap{\bm{\zeta}}{\hat{\bm{\alpha}}}} =
  \swap{\bm{\eta}.\bm{x}_{d}}{\bm{\beta}}^{%
    \swap{\bm{\zeta}}{\hat{\bm{\alpha}}}} =
  \swap{\bm{\eta}.\bm{x}_{d}}{\bm{\beta}}
  \]
  and the required equation then follows.

  (ii)~Note that our assumption that $\bm{\alpha} \perp \bm{\beta}$
  and $\bm{\beta} \in \Delta$ implies that $\bm{\beta} \perp
  \bm{\gamma}$.  Now consider first the case when $(m,k) = (3,2)$, so
  that $\bm{\gamma} \in \Delta$.  We simply conjugate
  Relation~\ref{rel:R5-welldef} by a permutation~$\sigma \in
  \Sym(\Delta)$ that moves $\bm{\delta}^{(0)}$ to~$\bm{\zeta}$,
  \ $\bm{\delta}^{(1)}$ to~$\bm{\beta}$ and $\bm{\delta}^{(2)}$
  to~$\bm{\gamma}$ (using relations established in the weight~$(3,1)$
  stage) to yield the required formula.

  Now consider the case when $(m,k) > (3,2)$.  Choose $\bm{\delta} \in
  \Delta$ that is incomparable with each of
  $\bm{\beta}$,~$\bm{\gamma}$ and~$\bm{\zeta}$.  Note that
  $\wt(\bm{\delta}.\bm{x}_{d}.\bm{y}_{d'}) = (3,2)$ and so the
  transpositions with this address as an entry in the following
  calculation were constructed at an earlier stage.  Then
  \[
  \swap{\bm{\zeta}.\bm{x}_{d}}{\bm{\beta}}^{%
    \swap{\bm{\zeta}}{\bm{\gamma}.\bm{y}_{d'}} \,
    \swap{\bm{\gamma}}{\bm{\delta}}} =
  \swap{\bm{\zeta}.\bm{x}_{d}}{\bm{\beta}}^{%
    \swap{\bm{\zeta}}{\bm{\delta}.\bm{y}_{d'}}} =
  \swap{\bm{\delta}.\bm{x}_{d}.\bm{y}_{d'}}{\bm{\beta}}
  \]
  and similarly $\swap{\bm{\zeta}.\bm{y}_{d'}}{\bm{\beta}}^{%
    \swap{\bm{\zeta}}{\bm{\gamma}.\bm{x}_{d}} \,
    \swap{\bm{\gamma}}{\bm{\delta}}} =
  \swap{\bm{\delta}.\bm{x}_{d}.\bm{y}_{d'}}{\bm{\beta}}$.  It
  therefore follows that the left-hand sides of these equations are
  equal, from which we deduce our required formula.
\end{proof}

It follows from part~(i) of this lemma that our
definition~\eqref{eq:inductionI-def}
of~$\swap{\bm{\alpha}}{\bm{\beta}}$ is independent of the choice
of~$\bm{\zeta}$.  Then part~(ii) shows the definition is also
independent of the choice of index~$d$.  In conclusion, the
transpositions~$\swap{\bm{\alpha}}{\bm{\beta}}$, where
$\wt(\bm{\alpha}) = (m,k)$ and $\bm{\beta} \in \Delta$ or \emph{vice
versa}, are well-defined.  The remaining work in this part of the
induction is to establish the four types of conjugacy
relations~\eqref{rel:conj} and then the split
relation~\eqref{rel:split} when they involve such transpositions.

\textbf{(A):} Consider two transpositions
$\swap{\bm{\alpha}}{\bm{\beta}}$
and~$\swap{\bm{\gamma}}{\bm{\delta}}$, at least one of which was
defined as in Equation~\eqref{eq:inductionI-def} and the other
possibly arriving at an early stage in the induction, such that every
pair of addresses from $\{ \bm{\alpha}, \bm{\beta}, \bm{\gamma},
\bm{\delta} \}$ is incomparable.  Exploiting the symmetry
relation~\eqref{rel:symm}, we can suppose without loss of generality
that one of the following sets of conditions holds:
\begin{enumerate}
  \renewcommand{\labelenumi}{(A.\roman{enumi})}
\item $\wt(\bm{\alpha}) = (m,k)$, \ $\bm{\beta} \in \Delta$ and
  $\wt(\bm{\gamma}), \wt(\bm{\delta}) < (m,k)$; or
\item $\wt(\bm{\alpha}) = \wt(\bm{\gamma}) = (m,k)$ and $\bm{\beta},
  \bm{\delta} \in \Delta$.
\end{enumerate}
In Case~(A.i), write $\bm{\alpha} = \hat{\bm{\alpha}}.\bm{x}_{d}$ as
above.  Choose~$\bm{\zeta} \in \Delta$ that is incomparable with each
of $\hat{\bm{\alpha}}$,~$\bm{\beta}$, $\bm{\gamma}$ and~$\bm{\delta}$
and use this, together with Lemma~\ref{lem:induct-welldef}, in the
definition~\eqref{eq:inductionI-def}
of~$\swap{\bm{\alpha}}{\bm{\beta}}$.  By induction,
$\swap{\bm{\gamma}}{\bm{\delta}}$~commutes with both
$\swap{\bm{\zeta}.\bm{x}_{d}}{\bm{\beta}}$
and~$\swap{\bm{\zeta}}{\hat{\bm{\alpha}}}$ and hence with
$\swap{\bm{\alpha}}{\bm{\beta}} =
\swap{\bm{\zeta}.\bm{x}_{d}}{\bm{\beta}}^{%
  \swap{\bm{\zeta}}{\hat{\bm{\alpha}}}}$, as required.  Case~(A.ii) is
similar, but we now write $\bm{\gamma} =
\hat{\bm{\gamma}}.\bm{y}_{d'}$ for some suitable index~$d'$ and choose
$\bm{\eta} \in \Delta$ incomparable with each address in $\{
\hat{\bm{\alpha}}, \bm{\beta}, \hat{\bm{\gamma}}, \bm{\delta},
\bm{\zeta} \}$.  We then observe that, by induction, each
transposition used in the definition $\swap{\bm{\gamma}}{\bm{\delta}}
= \swap{\bm{\eta}.\bm{y}_{d'}}{\bm{\delta}}^{%
  \swap{\bm{\eta}}{\hat{\bm{\gamma}}}}$ commutes with each one used to
construct~$\swap{\bm{\alpha}}{\bm{\beta}}$.  This establishes all
Type~(A) conjugacy relations at this stage.

\textbf{(B):} Consider a conjugacy relation~\eqref{rel:conj} where
$\bm{\alpha} = \bm{\gamma\eta}$ as in Lemma~\ref{lem:conj-types}(B).
At least one address in the relation has weight~$(m,k)$ and the other
entry in a transposition involving such an address must be
in~$\Delta$.  Exploiting the symmetry present, there are four
possibilities:
\begin{enumerate}
  \renewcommand{\labelenumi}{(B.\roman{enumi})}
\item $\wt(\bm{\beta}) = (m,k)$, \ $\bm{\gamma}, \bm{\delta} \in
  \Delta$ and $\bm{\eta} = \bm{\epsilon}$;
\item $\wt(\bm{\gamma\eta}) = (m,k)$, \ $\bm{\beta} \in \Delta$ and
  $\wt(\bm{\gamma}), \wt(\bm{\delta}), \wt(\bm{\delta\eta}) < (m,k)$;
\item $\wt(\bm{\gamma\eta}) = \wt(\bm{\delta\eta}) = (m,k)$,
  \ $\bm{\beta} \in \Delta$ and $\wt(\bm{\gamma}), \wt(\bm{\delta}) < (m,k)$; or
\item $\wt(\bm{\gamma}) = \wt(\bm{\gamma\eta}) = (m,k)$,
  \ $\bm{\beta}, \bm{\delta} \in \Delta$ and $\wt(\bm{\delta\eta}) <
  (m,k)$.
\end{enumerate}
(It is impossible that $\wt(\bm{\gamma}) = \wt(\bm{\gamma\eta}) =
\wt(\bm{\delta\eta}) = (m,k)$: If this situation were to happen, then
there would be some index~$d$ where the $d$th coordinate
of~$\bm{\delta\eta}$ has length~$m$.  Then the $d$th coordinate
of~$\bm{\gamma\eta}$ has length~$m$, but that of~$\bm{\gamma}$ is
shorter, contradicting $\wt(\bm{\gamma}) = \wt(\bm{\gamma\eta})$.)

In Case~(B.i), write $\bm{\beta} = \hat{\bm{\beta}}.\bm{x}_{d}$ and
choose $\bm{\zeta} \in \Delta$ incomparable with each of
$\hat{\bm{\beta}}$,~$\bm{\gamma}$ and~$\bm{\delta}$ in order to define
$\swap{\bm{\gamma}}{\bm{\beta}} =
\swap{\bm{\gamma}}{\bm{\zeta}.\bm{x}_{d}}^{%
  \swap{\bm{\zeta}}{\hat{\bm{\beta}}}}$ and similarly
for~$\swap{\bm{\delta}}{\bm{\beta}}$.  Then, by induction,
\[
\swap{\bm{\gamma\eta}}{\bm{\beta}}^{%
  \swap{\bm{\gamma}}{\bm{\delta}}} =
\swap{\bm{\gamma}}{\bm{\beta}}^{%
  \swap{\bm{\gamma}}{\bm{\delta}}} =
\swap{\bm{\gamma}}{\bm{\zeta}.\bm{x}_{d}}^{%
  \swap{\bm{\zeta}}{\hat{\bm{\beta}}} \,
  \swap{\bm{\gamma}}{\bm{\delta}}} =
\swap{\bm{\delta}}{\bm{\zeta}.\bm{x}_{d}}^{%
  \swap{\bm{\zeta}}{\hat{\bm{\beta}}}} =
\swap{\bm{\delta}}{\bm{\beta}}.
\]
In Case~(B.ii), note that there is an index~$d$ such that the $d$th
coordinate of~$\bm{\gamma\eta}$ has length~$m$ and that
of~$\bm{\gamma}$ is shorter.  Therefore we can write $\bm{\eta} =
\hat{\bm{\eta}}.\bm{x}_{d}$ and choose $\bm{\zeta} \in \Delta$
incomparable with each of $\bm{\beta}$,~$\bm{\gamma}$
and~$\bm{\delta}$ to define $\swap{\bm{\gamma\eta}}{\bm{\beta}}
= \swap{\bm{\zeta}.\bm{x}_{d}}{\bm{\beta}}^{%
  \swap{\bm{\zeta}}{\bm{\gamma}\hat{\bm{\eta}}}}$.  Then
\[
\swap{\bm{\gamma\eta}}{\bm{\beta}}^{%
  \swap{\bm{\gamma}}{\bm{\delta}}} =
\swap{\bm{\zeta}.\bm{x}_{d}}{\bm{\beta}}^{%
  \swap{\bm{\zeta}}{\bm{\gamma}\hat{\bm{\eta}}} \,
  \swap{\bm{\gamma}}{\bm{\delta}}} =
\swap{\bm{\zeta}.\bm{x}_{d}}{\bm{\beta}}^{%
  \swap{\bm{\zeta}}{\bm{\delta}\hat{\bm{\eta}}}} =
\swap{\bm{\delta}\hat{\bm{\eta}}.\bm{x}_{d}}{\bm{\beta}} =
\swap{\bm{\delta\eta}}{\bm{\beta}}.
\]
For Case~(B.iii), there are two possibilities.  If there is some~$d$
such that the $d$th coordinate of
$\bm{\gamma\eta}$~and~$\bm{\delta\eta}$ both have length~$m$ but those
of $\bm{\gamma}$~and~$\bm{\delta}$ are shorter, then we use the same
argument as for Case~(B.ii), but now the last step in the calculation
is actually the definition of~$\swap{\bm{\delta\eta}}{\bm{\beta}}$.
Otherwise, there are $d$~and~$d'$ such that the $d$th coordinate
of~$\bm{\gamma\eta}$ has length~$m$ and that of~$\bm{\gamma}$ is
shorter and the $d'$th coordinate of~$\bm{\delta\eta}$ has length~$m$
and that of~$\bm{\delta}$ is shorter.  Moreover, by hypothesis, the
$d$th coordinate of~$\bm{\gamma}$ must be longer than that
of~$\bm{\delta}$, so has length at least~$3$.  Choose distinct
addresses $\bm{\zeta}, \bm{\theta} \in \Delta$ incomparable with each
of $\bm{\beta}$,~$\bm{\gamma}$ and~$\bm{\delta}$.  We use~$\bm{\zeta}$
when employing Equation~\eqref{eq:inductionI-def} to define
$\swap{\bm{\gamma\eta}}{\bm{\beta}}$
and~$\swap{\bm{\delta\eta}}{\bm{\beta}}$, exploiting the coordinates
of indices $d$~and~$d'$, respectively, having written $\bm{\eta} =
\hat{\bm{\eta}}.\bm{x}_{d}.\bm{y}_{d'}$.  Thus
$\swap{\bm{\gamma\eta}}{\bm{\beta}} =
\swap{\bm{\zeta}.\bm{x}_{d}}{\bm{\beta}}^{%
  \swap{\bm{\zeta}}{\bm{\gamma}\hat{\bm{\eta}}.\bm{y}_{d'}}}$ and
similarly for~$\swap{\bm{\delta\eta}}{\bm{\beta}}$ (as in the second
set of calculations below).  Furthermore $\wt(\bm{\theta\eta}) <
(m,k)$ since the $d$th coordinate of~$\bm{\theta\eta}$ is shorter than
that of~$\bm{\gamma\eta}$.  We therefore compute:
\begin{align*}
  \swap{\bm{\gamma\eta}}{\bm{\beta}}^{%
    \swap{\bm{\gamma}}{\bm{\delta}} \,
    \swap{\bm{\gamma}}{\bm{\theta}}}
  &=
  \swap{\bm{\zeta}.\bm{x}_{d}}{\bm{\beta}}^{%
    \swap{\bm{\zeta}}{\bm{\gamma}\hat{\bm{\eta}}.\bm{y}_{d'}} \,
    \swap{\bm{\gamma}}{\bm{\delta}} \,
    \swap{\bm{\gamma}}{\bm{\theta}}} \\
  &=
  \swap{\bm{\zeta}.\bm{x}_{d}}{\bm{\beta}}^{%
    \swap{\bm{\zeta}}{\bm{\theta}\hat{\bm{\eta}}.\bm{y}_{d'}} \,
    \swap{\bm{\theta}}{\bm{\delta}}} \\
  &=
  \swap{\bm{\theta}\hat{\bm{\eta}}.\bm{x}_{d}.\bm{y}_{d'}}{\bm{\beta}}^{%
    \swap{\bm{\theta}}{\bm{\delta}}} =
  \swap{\bm{\theta\eta}}{\bm{\beta}}^{%
    \swap{\bm{\theta}}{\bm{\delta}}} \\[7pt]
  \swap{\bm{\delta\eta}}{\bm{\beta}}^{%
    \swap{\bm{\theta}}{\bm{\delta}}} &=
  \swap{\bm{\zeta}.\bm{y}_{d'}}{\bm{\beta}}^{%
    \swap{\bm{\zeta}}{\bm{\delta}\hat{\bm{\eta}}.\bm{x}_{d}} \,
    \swap{\bm{\theta}}{\bm{\delta}}} \\
  &=
  \swap{\bm{\zeta}.\bm{y}_{d'}}{\bm{\beta}}^{%
    \swap{\bm{\zeta}}{\bm{\theta}\hat{\bm{\eta}}.\bm{x}_{d}}} \\
  &= \swap{\bm{\theta\eta}}{\bm{\beta}}
\end{align*}
Hence $\swap{\bm{\gamma\eta}}{\bm{\beta}}^{%
  \swap{\bm{\gamma}}{\bm{\delta}} \, \swap{\bm{\gamma}}{\bm{\theta}}}
= \swap{\bm{\delta\eta}}{\bm{\beta}}$ and, with use of our already
verified Type~(A) conjugacy relation, we conclude
$\swap{\bm{\gamma\eta}}{\bm{\beta}}^{%
  \swap{\bm{\gamma}}{\bm{\delta}}} =
\swap{\bm{\delta\eta}}{\bm{\beta}}^{%
  \swap{\bm{\gamma}}{\bm{\theta}}} =
\swap{\bm{\delta\eta}}{\bm{\beta}}$.

Finally, in Case~(B.iv), let $d$~be the index of a coordinate
of~$\bm{\gamma}$ of length~$m$.  Write $\bm{\gamma} =
\hat{\bm{\gamma}}.\bm{x}_{d}$.  Note that the $d$th coordinate
of~$\bm{\eta}$ is empty.  Choose distinct addresses $\bm{\zeta},
\bm{\theta} \in \Delta$ that are incomparable with each of
$\bm{\beta}$,~$\hat{\bm{\gamma}}$ and~$\bm{\delta}$.  The first is
used in the construction of $\swap{\bm{\gamma}}{\bm{\delta}}$
and~$\swap{\bm{\gamma\eta}}{\bm{\beta}}$.  One observes that
$\wt(\bm{\theta\eta}.\bm{x}_{d}) < (m,k)$.  We calculate:
\[
\swap{\bm{\gamma}}{\bm{\delta}}^{%
  \swap{\hat{\bm{\gamma}}}{\bm{\theta}}} =
\swap{\bm{\zeta}.\bm{x}_{d}}{\bm{\delta}}^{%
  \swap{\bm{\zeta}}{\hat{\bm{\gamma}}} \,
  \swap{\hat{\bm{\gamma}}}{\bm{\theta}}} =
\swap{\bm{\zeta}.\bm{x}_{d}}{\bm{\delta}}^{%
  \swap{\hat{\bm{\gamma}}}{\bm{\theta}} \,
  \swap{\bm{\zeta}}{\bm{\theta}}} =
\swap{\bm{\zeta}.\bm{x}_{d}}{\bm{\delta}}^{%
  \swap{\bm{\zeta}}{\bm{\theta}}} =
\swap{\bm{\theta}.\bm{x}_{d}}{\bm{\delta}}
\]
(since $\swap{\bm{\zeta}.\bm{x}_{d}}{\bm{\delta}}$ and
$\swap{\hat{\bm{\gamma}}}{\bm{\theta}}$ commute) so that
\begin{align*}
  \swap{\bm{\gamma\eta}}{\bm{\beta}}^{%
    \swap{\bm{\gamma}}{\bm{\delta}} \,
    \swap{\hat{\bm{\gamma}}}{\bm{\theta}}}
  =
  \swap{\bm{\zeta}.\bm{x}_{d}}{\bm{\beta}}^{%
    \swap{\bm{\zeta}}{\hat{\bm{\gamma}}\bm{\eta}} \,
    \swap{\bm{\gamma}}{\bm{\delta}} \,
    \swap{\hat{\bm{\gamma}}}{\bm{\theta}}}
  &=
  \swap{\bm{\zeta}.\bm{x}_{d}}{\bm{\beta}}^{%
    \swap{\bm{\zeta}}{\bm{\theta\eta}} \,
    \swap{\bm{\theta}.\bm{x}_{d}}{\bm{\delta}}} \\
  &=
  \swap{\bm{\theta\eta}.\bm{x}_{d}}{\bm{\beta}}^{%
    \swap{\bm{\theta}.\bm{x}_{d}}{\bm{\delta}}}
  =
  \swap{\bm{\delta\eta}}{\bm{\beta}}.
\end{align*}
Hence, by a Type~(A) conjugacy relation,
$\swap{\bm{\gamma\eta}}{\bm{\beta}}^{%
  \swap{\bm{\gamma}}{\bm{\delta}}} =
\swap{\bm{\delta\eta}}{\bm{\beta}}^{%
  \swap{\hat{\bm{\gamma}}}{\bm{\theta}}} =
\swap{\bm{\delta\eta}}{\bm{\beta}}$.  This establishes all the
Type~(B) conjugacy relations.

\textbf{(C):}~In the notation of Lemma~\ref{lem:conj-types}(C), if it
were the case that $\wt(\bm{\gamma\eta}) = \wt(\bm{\gamma\theta}) =
(m,k)$, then at this stage $\bm{\delta}, \bm{\delta\eta},
\bm{\delta\theta} \in \Delta$, which would force $\bm{\eta} =
\bm{\theta} = \bm{\epsilon}$.  The conjugacy relation would reduce to
one form~$g^{g} = g$ that holds in any group.  Consequently, upon
exploiting the symmetry in the relation, we must
verify~\eqref{rel:conj} in the following two cases:
\begin{enumerate}
  \renewcommand{\labelenumi}{(C.\roman{enumi})}
\item $\wt(\bm{\gamma\eta}) = (m,k)$, \ $\bm{\delta} \in \Delta$,
  \ $\bm{\theta} = \bm{\epsilon}$ and $\wt(\bm{\gamma}) < (m,k)$; or
\item $\wt(\bm{\gamma\eta}) = \wt(\bm{\delta\eta}) = (m,k)$,
  \ $\bm{\gamma}, \bm{\delta} \in \Delta$ and $\bm{\theta} =
  \bm{\epsilon}$.
\end{enumerate}
Both are dealt with in the same manner, namely by an argument
similar to Case~(B.ii) above.

\textbf{(D)/Split:} Note that, in the notation of
Lemma~\ref{lem:conj-types}(D), the addresses
$\bm{\eta}$~and~$\bm{\theta}$ must be non-empty.  However, since all
the tranpositions introduced via~\eqref{eq:inductionI-def} have one
entry from~$\Delta$, we conclude that no new conjugacy relations of
Type~(D) must be verified at this stage.  Similarly there are no split
relations~\eqref{rel:split} to verify at this stage.  In conclusion,
we have verified all the required relations involving the
transpositions that have been introduced.

\paragraph{Induction, Stage~2:} At the second stage of the induction
we assume that, for the fixed weight $(m,k) \geq (3,1)$, we have
already defined all transpositions~$\swap{\bm{\alpha}}{\bm{\beta}}$
where $\bm{\alpha}, \bm{\beta} \in \Omega^{\ast}$ satisfy $\bm{\alpha}
\perp \bm{\beta}$ and either $\wt(\bm{\alpha}), \wt(\bm{\beta}) <
(m,k)$, or $\wt(\bm{\alpha}) = (m,k)$ and~$\bm{\beta} \in \Delta$ (or
\emph{vice versa}).  The former case holds by the inductive assumption
and the latter by the completion of Stage~1.  We also assume that we
have verified all relations~\eqref{rel:order}--\eqref{rel:symm}
involving these transpositions.  We now define the remaining
transpositions with entries of weight at most~$(m,k)$.

Assume then that $\bm{\alpha}$~and~$\bm{\beta}$ are incomparable
addresses in~$\Omega^{\ast}$ of which one has weight~$(m,k)$ and the
other has weight at most~$(m,k)$ and is not from~$\Delta$.  Choose
$\bm{\zeta} \in \Omega$ incomparable with both
$\bm{\alpha}$~and~$\bm{\beta}$ and define
\begin{equation}
  \swap{\bm{\alpha}}{\bm{\beta}} \defeq
  \swap{\bm{\alpha}}{\bm{\zeta}}^{%
    \swap{\bm{\beta}}{\bm{\zeta}}}.
  \label{eq:stage2-def}
\end{equation}
At least one of the transpositions on the right-hand side is defined
via Stage~1, while the other (in the case that the relevant entry has
weight~$< (m,k)$) may have been constructed earlier in the inductive
process.  We first verify that this definition is independent of the
choice of~$\bm{\zeta}$.

\begin{lemma}
  \label{lem:stage2-welldef}
  Let $\bm{\alpha}$,~$\bm{\beta}$, $\bm{\zeta}$ and~$\bm{\eta}$ be
  incomparable addresses in~$\Omega^{\ast}$ with $\wt(\bm{\alpha}),
  \wt(\bm{\beta}) \leq (m,k)$ and $\bm{\zeta}, \bm{\eta} \in \Delta$.
  Then
  \begin{enumerate}
  \item $\swap{\bm{\alpha}}{\bm{\zeta}}^{%
    \swap{\bm{\beta}}{\bm{\zeta}}} = \swap{\bm{\alpha}}{\bm{\eta}}^{%
    \swap{\bm{\beta}}{\bm{\eta}}}$;
  \item $\swap{\bm{\alpha}}{\bm{\zeta}}^{%
    \swap{\bm{\beta}}{\bm{\zeta}}} = \swap{\bm{\beta}}{\bm{\zeta}}^{%
    \swap{\bm{\alpha}}{\bm{\zeta}}}$.
  \end{enumerate}
\end{lemma}

\begin{proof}
  (i)~In the following calculation, and indeed for many used during
  this stage, all the transpositions we manipulate involve one entry
  from~$\Delta$.  Hence the relations we rely upon hold by induction
  or were established in Stage~1.  Observe
  \[
  \swap{\bm{\alpha}}{\bm{\zeta}}^{%
    \swap{\bm{\beta}}{\bm{\zeta}} \, \swap{\bm{\beta}}{\bm{\eta}}} =
  \swap{\bm{\alpha}}{\bm{\zeta}}^{%
    \swap{\bm{\beta}}{\bm{\eta}} \, \swap{\bm{\eta}}{\bm{\zeta}}} =
  \swap{\bm{\alpha}}{\bm{\zeta}}^{%
    \swap{\bm{\eta}}{\bm{\zeta}}} =
  \swap{\bm{\alpha}}{\bm{\eta}},
  \]
  from which the claimed equation follows.

  (ii)~Choose $\bm{\eta} \in \Delta$ that is incomparable with each of
  $\bm{\alpha}$,~$\bm{\beta}$ and~$\bm{\zeta}$.  We then calculate
  \[
  \swap{\bm{\alpha}}{\bm{\zeta}}^{%
    \swap{\bm{\beta}}{\bm{\zeta}} \, \swap{\bm{\beta}}{\bm{\eta}}} =
  \swap{\bm{\alpha}}{\bm{\zeta}}^{%
    \swap{\bm{\eta}}{\bm{\zeta}}} =
  \swap{\bm{\alpha}}{\bm{\eta}}
  \AND
  \swap{\bm{\beta}}{\bm{\zeta}}^{%
    \swap{\bm{\alpha}}{\bm{\zeta}} \, \swap{\bm{\beta}}{\bm{\eta}}} =
  \swap{\bm{\eta}}{\bm{\zeta}}^{%
    \swap{\bm{\alpha}}{\bm{\zeta}}} =
  \swap{\bm{\eta}}{\bm{\alpha}},
  \]
  so that $\swap{\bm{\alpha}}{\bm{\zeta}}^{%
    \swap{\bm{\beta}}{\bm{\zeta}} \, \swap{\bm{\beta}}{\bm{\eta}}} =
  \swap{\bm{\beta}}{\bm{\zeta}}^{%
    \swap{\bm{\alpha}}{\bm{\zeta}} \, \swap{\bm{\beta}}{\bm{\eta}}}$,
  from which the claimed equation follows.
\end{proof}

Part~(i) of this lemma tells us that the definition~\eqref{eq:stage2-def} of~$\swap{\bm{\alpha}}{\bm{\beta}}$ is
independent of the choice of $\bm{\zeta} \in \Delta$.  Part~(ii) tells
us that $\swap{\bm{\alpha}}{\bm{\beta}} =
\swap{\bm{\beta}}{\bm{\alpha}}$; that is, Relation~\eqref{rel:symm}
holds for the transpositions defined via~\eqref{eq:stage2-def}.

Before verifying the remaining relations, we shall observe that
\[
\swap{\bm{\alpha}}{\bm{\beta}} = \swap{\bm{\alpha}}{\bm{\zeta}}^{%
  \swap{\bm{\beta}}{\bm{\zeta}}}
\]
holds for every triple $\bm{\alpha}$,~$\bm{\beta}$ and~$\bm{\zeta}$ of
pairwise incomparable addresses in~$\Omega^{\ast}$ with
$\wt(\bm{\alpha}), \wt(\bm{\beta}) \leq (m,k)$ and $\bm{\zeta} \in
\Delta$.  When one or both of $\bm{\alpha}$~and~$\bm{\beta}$ have
weight~$(m,k)$ this is the definition~\eqref{eq:stage2-def} combined
with Lemma~\ref{lem:stage2-welldef}(i).  When they both have weight~$<
(m,k)$, it follows by induction.

The four Cases~(A)--(D) of conjugacy relations~\eqref{rel:conj}
described in Lemma~\ref{lem:conj-types} are all established by the
same method.  We illustrate this for~(B), namely we establish
\[
\swap{\bm{\gamma\eta}}{\bm{\beta}}^{%
  \swap{\bm{\gamma}}{\bm{\delta}}} =
\swap{\bm{\delta\eta}}{\bm{\beta}}
\]
for incomparable addresses $\bm{\beta}, \bm{\gamma}, \bm{\delta} \in
\Omega^{\ast}$ and some (possibly empty) $\bm{\eta} \in \Omega$ such
that the addresses appearing in the formula all have weight~$\leq
(m,k)$.  Choose distinct $\bm{\zeta}, \bm{\theta} \in \Delta$
incomparable with each of $\bm{\beta}$,~$\bm{\gamma}$
and~$\bm{\delta}$, so that $\swap{\bm{\gamma\eta}}{\bm{\beta}} =
\swap{\bm{\gamma\eta}}{\bm{\zeta}}^{\swap{\bm{\beta}}{\bm{\zeta}}}$
and $\swap{\bm{\gamma}}{\bm{\delta}} =
\swap{\bm{\gamma}}{\bm{\theta}}^{\swap{\bm{\delta}}{\bm{\theta}}}$.
In the following calculation, all the transpositions manipulated have
second entry either $\bm{\zeta}$~or~$\bm{\theta}$ (selected
from~$\Delta$):
\begin{align*}
  \swap{\bm{\gamma\eta}}{\bm{\beta}}^{%
    \swap{\bm{\gamma}}{\bm{\delta}}} =
  \swap{\bm{\gamma\eta}}{\bm{\zeta}}^{%
    \swap{\bm{\beta}}{\bm{\zeta}} \, \swap{\bm{\delta}}{\bm{\theta}} \,
    \swap{\bm{\gamma}}{\bm{\theta}} \, \swap{\bm{\delta}}{\bm{\theta}}}
  &= \swap{\bm{\gamma\eta}}{\bm{\zeta}}^{%
    \swap{\bm{\beta}}{\bm{\zeta}} \, \swap{\bm{\gamma}}{\bm{\theta}} \,
    \swap{\bm{\delta}}{\bm{\theta}}} \\
  &= \swap{\bm{\theta\eta}}{\bm{\zeta}}^{%
    \swap{\bm{\beta}}{\bm{\zeta}} \, \swap{\bm{\delta}}{\bm{\theta}}} =
  \swap{\bm{\delta\eta}}{\bm{\zeta}}^{%
    \swap{\bm{\beta}}{\bm{\zeta}}}
\end{align*}
and the latter is equal to~$\swap{\bm{\delta\eta}}{\bm{\beta}}$.
Cases~(A), (C) and~(D) of the conjugacy relations are established
similarly.

Finally we establish all split relations~\eqref{rel:split} for this
stage.  If $(m,k) = (3,1)$, then an arbitrary split relation has the
form
\[
\swap{\bm{\alpha}}{\bm{\beta}} =
\swap{\bm{\alpha}.\bm{0}_{d}}{\bm{\beta}.\bm{0}_{d}} \,
\swap{\bm{\beta}.\bm{1}_{d}}{\bm{\beta}.\bm{1}_{d}}
\]
for incomparable $\bm{\alpha}, \bm{\beta} \in \Delta$.  This is
deduced from Relation~\ref{rel:R6-split} by conjugating by a
permutation $\sigma \in \Sym(\Delta)$ that moves $\bm{\delta}^{(0)}$
to~$\bm{\alpha}$ and $\bm{\delta}^{(1)}$ to~$\bm{\beta}$.  For the
case when $(m,k) > (3,1)$, choose $\bm{\zeta}, \bm{\eta} \in \Delta$
such that every pair from $\{ \bm{\alpha}, \bm{\beta}, \bm{\zeta},
\bm{\eta} \}$ are incomparable.  We have just established that
$\swap{\bm{\zeta}}{\bm{\eta}} =
\swap{\bm{\zeta}.\bm{0}_{d}}{\bm{\eta}.\bm{0}_{d}} \,
\swap{\bm{\zeta}.\bm{1}_{d}}{\bm{\eta}.\bm{1}_{d}}$ and we deduce the
general case of~\eqref{rel:split} by conjugating
by~$\swap{\bm{\zeta}}{\bm{\alpha}} \, \swap{\bm{\eta}}{\bm{\beta}}$
and using the conjugacy relations~\eqref{rel:conj} that we have
already established.

\paragraph{Transpositions with short coordinates:} We have now
constructed all transpositions~$\swap{\bm{\alpha}}{\bm{\beta}}$ with
$\bm{\alpha}, \bm{\beta} \in \Omega^{\ast}$ and $\bm{\alpha} \perp
\bm{\beta}$ and demonstrated that all the required
relations~\eqref{rel:order}--\eqref{rel:split} (and
also~\eqref{rel:symm}) involving such transpositions hold in~$G$.  We
complete the definitions by constructing the remaining
transpositions~$\swap{\bm{\alpha}}{\bm{\beta}}$ with $\bm{\alpha},
\bm{\beta} \in \Omega$.  Fix a sequence $(k_{1},k_{2},\dots,k_{n})$
with each $k_{i} \in \{0,1,2\}$.  As an induction hypothesis, assume
that we have constructed all
transpositions~$\swap{\bm{\alpha}}{\bm{\beta}}$ where $\bm{\alpha} =
(\alpha_{1},\alpha_{2},\dots,\alpha_{n})$ and $\bm{\beta} =
(\beta_{1},\beta_{2},\dots,\beta_{n})$ are incomparable addresses
satisfying $\length{\alpha_{i}}, \length{\beta_{i}} \geq k_{i}$ for $1
\leq i \leq d$.  (The ``base case'' is $(k_{1},k_{2},\dots,k_{n}) =
(2,2,\dots,2)$, for which our assumption follows from the steps just
completed.)  Select an index~$d$ with $k_{d} > 0$ and for any pair
of incomparable addresses $\bm{\alpha}$~and~$\bm{\beta}$ such that
$\length{\alpha_{i}}, \length{\beta_{i}} \geq k_{i}$ for $i \neq d$
and such that one, or possibly both, of $\alpha_{d}$~or~$\beta_{d}$
has length~$k_{d}-1$, define
\begin{equation}
  \swap{\bm{\alpha}}{\bm{\beta}} \defeq
  \swap{\bm{\alpha}.\bm{0}_{d}}{\bm{\beta}.\bm{0}_{d}} \,
  \swap{\bm{\alpha}.\bm{1}_{d}}{\bm{\beta}.\bm{1}_{d}}.
  \label{eq:shortdef}
\end{equation}
Both transpositions on the right-hand side exist by our assumption.
Furthermore, the transpositions on the right-hand satisfy the
relations~\eqref{rel:order} and~\eqref{rel:symm} and commute.  It
follows that $\swap{\bm{\alpha}}{\bm{\beta}} =
\swap{\bm{\beta}}{\bm{\alpha}}$ and
$\swap{\bm{\alpha}}{\bm{\beta}}^{2} = 1$.  Conjugacy relations are
established similarly to earlier steps, namely by considering
Cases~(A)--(D) of Lemma~\ref{lem:conj-types}.  The method is the same
for each case.  For example, consider a Case~(B) conjugacy relation;
that is, one of the form
\[
\swap{\bm{\gamma\eta}}{\bm{\beta}}^{%
  \swap{\bm{\gamma}}{\bm{\delta}}} =
\swap{\bm{\delta\eta}}{\bm{\beta}}
\]
where some entry here has its $d$th coordinate of length~$k_{d}-1$.
We may assume the entry with this
shorter coordinate is either $\bm{\gamma}$ (and possibly
also~$\bm{\gamma\eta}$) or~$\bm{\beta}$.  If the $d$th coordinate
of~$\bm{\gamma}$ has length~$k_{d}-1$ and that of~$\bm{\eta}$ is
empty, then we use the formula~\eqref{eq:shortdef} for both
$\swap{\bm{\gamma\eta}}{\bm{\beta}}$ and~$\swap{\bm{\gamma}}{\bm{\delta}}$.  Note then
$\bm{\eta}.\bm{x}_{d} = \bm{x}_{d}.\bm{\eta}$ for $x \in \{0,1\}$,
which permits us to calculate the following conjugate:
\begin{align*}
  \swap{\bm{\gamma\eta}}{\bm{\beta}}^{%
    \swap{\bm{\gamma}}{\bm{\delta}}} &=
  \bigl(
  \swap{\bm{\gamma\eta}.\bm{0}_{d}}{\bm{\beta}.\bm{0}_{d}} \,
  \swap{\bm{\gamma\eta}.\bm{1}_{d}}{\bm{\beta}.\bm{1}_{d}}
  \bigr)^{%
    \swap{\bm{\gamma}.\bm{0}_{d}}{\bm{\delta}.\bm{0}_{d}} \,
    \swap{\bm{\gamma}.\bm{1}_{d}}{\bm{\delta}.\bm{1}_{d}}} \\
  &= \swap{\bm{\delta\eta}.\bm{0}_{d}}{\bm{\beta}.\bm{0}_{d}} \,
  \swap{\bm{\delta\eta}.\bm{1}_{d}}{\bm{\beta}.\bm{1}_{d}}
  = \swap{\bm{\delta\eta}}{\bm{\beta}}.
\end{align*}
relying upon relations that hold by the inductive assumption.  The
last step is either one of these assumed relations or is the
definition of~$\swap{\bm{\delta\eta}}{\bm{\beta}}$ if it is the case
that the $d$th coordinate of $\bm{\beta}$~or~$\bm{\delta}$ has
length~$k_{d}-1$.  Alternatively if the $d$th coordinate
of~$\bm{\gamma}$ has length~$k_{d}-1$ and that of~$\bm{\eta}$ is
non-empty, write $\bm{\eta} = \bm{x}_{d}.\hat{\bm{\eta}}$ for some $x
\in \{0,1\}$ and some (possibly empty) $\hat{\bm{\eta}} \in \Omega$.
In this case, we use Equation~\eqref{eq:shortdef} for the definition
of~$\swap{\bm{\gamma}}{\bm{\delta}}$ and calculate
\begin{align*}
  \swap{\bm{\gamma\eta}}{\bm{\beta}}^{%
    \swap{\bm{\gamma}}{\bm{\delta}}} &=
  \swap{\bm{\gamma}.\bm{x}_{d}.\hat{\bm{\eta}}}{\bm{\beta}}^{%
    \swap{\bm{\gamma}.\bm{0}_{d}}{\bm{\delta}.\bm{0}_{d}} \,
    \swap{\bm{\gamma}.\bm{1}_{d}}{\bm{\delta}.\bm{1}_{d}}} \\
  &= \swap{\bm{\delta}.\bm{x}_{d}.\hat{\bm{\eta}}}{\bm{\beta}} =
  \swap{\bm{\delta\eta}}{\bm{\beta}}.
\end{align*}
Conjugacy relations in Case~(B) where the $d$th coordinate
of~$\bm{\beta}$ has length~$k_{d}-1$ and those in Cases~(A), (C)
and~(D) are established similarly.  Finally the split
relations~\eqref{rel:split} involving the transpositions defined
in~\eqref{eq:shortdef} are either simply that definition or are
inherited from split relations for the terms on the right-hand side of
that formula.

It now follows, using this step repeatedly, that we have constructed
transpositions~$\swap{\bm{\alpha}}{\bm{\beta}}$, for $\bm{\alpha},
\bm{\beta} \in \Omega$ with $\bm{\alpha} \perp \bm{\beta}$, in the
group~$G$ and verified all
relations~\eqref{rel:order}--\eqref{rel:split} involving these
transpositions.  Consequently, by Theorem~\ref{thm:infpres}, there is
a homomorphism $\phi \colon nV \to G$ mapping each transposition
in~$nV$ to the corresponding element that we have defined in~$G$.
Moreover, Relation~\ref{rel:R7-c-def} tells us that the generator~$c$
is in the image of~$\phi$ and hence $G$~is isomorphic to a quotient
of~$nV$.  On the other hand, all relations
\ref{rel:R1-Sym}--\ref{rel:R7-c-def} listed are satisfied by the
corresponding elements of~$nV$ and so there is a homomorphism from~$G$
into~$nV$ with non-trivial image.  The fact that $nV$~is simple
therefore yields $G \cong nV$, completing the proof of
Theorem~\ref{thm:finitepres}.

\begin{trivlist}
\item\textsc{Proof of Corollary~\ref{cor:2gen}:} The subgroup $H =
  \langle a,b \rangle \cong \Sym(\Delta)$ of~$G$ can be generated by a
  cycle~$x$ of length~$4^{n}$ and a transposition~$t$ that can be
  assumed disjoint from~$c$ (as described via
  Relation~\ref{rel:R7-c-def}).  Note that $c$~has odd order.
  Therefore $c$~and~$t$ are powers of~$y = ct$ and $\{ x, y \}$~is a
  generating set for~$G$.  Applying Tietze transformations to produce
  a presentation on generators $x$~and~$y$ introduces two additional
  relations.  This establishes the corollary. \qed
\end{trivlist}

\paragraph{Acknowledgements:} The author thanks Collin Bleak for
introducing him to Brin's family of groups~$nV$, suggesting this
project for investigation and for various helpful comments.  It should
also be recorded that the fact that $nV$~can be generated by two
elements, as appears in the proof of Corollary~\ref{cor:2gen}, was
originally observed in collaboration with Bleak.  The author also
thanks Prof.\ Brin for his comments on a draft of this article and a
referee of an earlier version for their helpful suggestions.


\begin{thebibliography}{99}

\bibitem{BL}
  Collin Bleak \& Daniel Lanoue, ``A family of non-isomorphism
  results,'' \textit{Geom.\ Dedicata}~\textbf{146} (2010), 21--26.

\bibitem{BQ}
  Collin Bleak \& Martyn Quick, ``The infinite simple group~$V$ of
  Richard~J. Thompson: presentations by permutations,'' \textit{Groups
    Geom.\ Dyn.}~\textbf{11} (2017), 1401--1436.

\bibitem{Brin-higherdim}
  Matthew~G. Brin, ``Higher dimensional Thompson groups,''
  \textit{Geom.\ Dedicata}~\textbf{108} (2004), 163--192.

\bibitem{Brin-presentations}
  Matthew~G. Brin, ``Presentations of higher dimensional Thompson
  groups,'' \textit{J. Algebra}~\textbf{284} (2005), 520--558.

\bibitem{Brin-baker}
  Matthew~G. Brin, ``On the baker's map and the simplicity of the
  higher dimensional Thompson groups~$nV$,''
  \textit{Publ.\ Mat.}~\textbf{54} (2010), no.~2, 433--439.

\bibitem{Brouwer}
  L.~E.~J. Brouwer, ``On the structure of perfect sets of points,''
  \textit{Proc.\ Akad.\ Amsterdam}~\textbf{12} (1910), 785--794.

\bibitem{CFP}
  J.~W. Cannon, W.~J. Floyd \& W.~R. Parry, ``Introductory notes on
  Richard Thompson's groups,''
  \textit{Enseign.\ Math.~(2)}~\textbf{42} (1996), no.~3--4,
  215--256.

\bibitem{Dehornoy}
  Patrick Dehornoy, ``Geometric presentations for Thompson's groups,''
  \textit{J. Pure Appl.\ Algebra}~\textbf{203} (2005), no.~1--3,
  1--44.

\bibitem{Devaney}
  Robert~L. Devaney, \textit{An introduction to chaotic dynamical
    systems, Second edition}, Addison-Wesley Studies in Nonlinearity,
  Addison-Wesley Publishing Company, Advanced Book Program, Redwood
  City, CA, 1989.

\bibitem{FAZR}
  Gustavo~A. Fern\'{a}ndez-Alcober \& Amaia Zugadi-Reizabal,
  ``GGS-groups: order of congruence quotients and Hausdorff
  dimension,'' \textit{Trans.\ Amer.\ Math.\ Soc.}~\textbf{366}
  (2014), no.~4, 1993--2017.

\bibitem{FMWZ}
  Martin~G. Fluch, Marco Marschler, Stefan Witzel \& Matthew
  C.~B. Zaremsky, ``The Brin--Thompson groups~$sV$ are of
  type~$\mathrm{F}_{\infty}$,'' \textit{Pacific J. Math.}~\textbf{266}
  (2013), no.~2, 283--295.

\bibitem{Grig80}
  R.~I. Grigorchuk, ``On Burnside's problem on periodic groups,''
  \textit{Funktsional.\ Anal.\ i Prilozhen.}~\textbf{14} (1980),
  no.~1, 53--54.

\bibitem{Grig91}
  Rostislav~I. Grigorchuk, ``On growth in group theory,''
  \textit{Proceedings of the International Congress of Mathematicians,
    Vols.~I,~II (Kyoto)}, 325--338, Math.\ Soc.\ Japan, Tokyo, 1991.

\bibitem{GNS}
  R.~I. Grigorchuk, V.~V. Nekrashevich \& V.~I. Sushchanski\u{\i},
  ``Automata, dynamical systems, and groups,''
  \textit{Tr.\ Mat.\ Inst.\ Steklova}~\textbf{231} (2000), 134--214;
  translation in \textit{Proc.\ Steklov Inst.\ Math.}~\textbf{213}
  (2000), no.~4, 128--203.

\bibitem{GS83}
  Narain Gupta \& Sa\"{\i}d Sidki, ``On the Burnside problem for
  periodic groups,'' \textit{Math.\ Z.}~\textbf{182} (1983), no.~3,
  385--388.

\bibitem{GKKL}
  R.~M. Guralnick, W.~M. Kantor, M. Kassabov \& A. Lubotzky,
  ``Presentations of finite simple groups: a computational approach,''
  \textit{J. Eur.\ Math.\ Soc.\ (JEMS)}~\textbf{13} (2011), no.~2,
  391--458.

\bibitem{HM}
  Johanna Hennig \& Francesco Matucci, ``Presentations for the
  higher-dimensional Thompson groups~$nV$,'' \textit{Pacific
    J. Math.}~\textbf{257} (2012), no.~1, 53--74.

\bibitem{Huxford}
  Peter Huxford, ``Short presentations of finite simple groups''
  (preprint), Jul 2019. \texttt{arXiv: 1907.10828}

\bibitem{KMPN}
  Dessislava~H. Kochloukova, Conchita Mart\'{\i}nez-P\'{e}rez \& Brita
  E.~A. Nucinkas, ``Cohomological finiteness properties of the
  Brin--Thompson--Higman groups $2V$~and~$3V$,''
  \textit{Proc.\ Edinb.\ Math.\ Soc.~(2)}~\textbf{56} (2013), no.~3,
  777--804.

\bibitem{MPMN}
  C. Martinez-P\'{e}rez, F.~Matucci \& B.~Nucinkis, ``Presentations of
  generalisations of Thompson's group~$V$,'' \textit{Pacific
    J. Math.}~\textbf{296} (2018), no.~2, 371--403.

\bibitem{Thompson-notes}
  Richard~J. Thompson, Handwritten widely circulated notes, 1965.

\end{thebibliography}
\end{document}